\documentclass[12pt]{amsart}
\usepackage[margin=0.5in]{geometry}
\usepackage{graphicx, epstopdf,color}
\usepackage{amssymb,amscd,amsthm,amsxtra}
 
\usepackage{mathrsfs}
\renewcommand{\mathcal}{\mathscr}

\numberwithin{equation}{section}

\newtheorem{theorem}{Theorem}
\newtheorem{lemma}[theorem]{Lemma}
\newtheorem{proposition}[theorem]{Proposition}

\newtheorem{corollary}[theorem]{Corollary}

\newcommand{\R}{\mathbb R}
\newcommand{\N}{\mathbb N}

\newcommand{\Z}{\mathbb Z} 

\renewcommand{\leq}{\leqslant}
\renewcommand{\le}{\leqslant}
\renewcommand{\geq}{\geqslant}
\renewcommand{\ge}{\geqslant}
\renewcommand{\epsilon}{\varepsilon}
\newcommand{\eps}{\varepsilon}

\begin{document}

\title[Nonlocal Delaunay surfaces]{Nonlocal Delaunay surfaces}

\author[J. D{\'a}vila]{Juan D{\'a}vila}
\author[M. del Pino]{Manuel del Pino}
\author[S. Dipierro]{Serena Dipierro}
\author[E. Valdinoci]{Enrico Valdinoci}

\address[Juan D{\'a}vila and Manuel del Pino]{Departamento de Ingenier\'ia
Matem\'atica and Centro de Modelamiento Matem\'atico,
Universidad de Chile,
Casilla 170 Correo 3,
8370459 Santiago,
Chile}
\email{jdavila@dim.uchile.cl, delpino@dim.uchile.cl}

\address[Serena Dipierro]{Institut f\"ur Analysis und Numerik, 
Otto-von-Guericke-Universit\"at Magdeburg,
Universit\"atsplatz 2, 39106 Magdeburg, Germany}
\email{serena.dipierro@ed.ac.uk}

\address[Enrico Valdinoci]{ Weierstra{\ss} Institut f\"ur Angewandte
Analysis und Stochastik, Hausvogteiplatz 11A, 10117 Berlin, Germany}
\email{enrico.valdinoci@wias-berlin.de}

\begin{abstract}
We construct codimension~$1$ surfaces of any dimension
that minimize
a periodic
nonlocal perimeter functional among surfaces that are periodic, cylindrically
symmetric and decreasing.

These surfaces may be seen as a nonlocal analogue of the classical
Delaunay surfaces (onduloids).
For small volume, most of their mass tends to be
concentrated in a periodic array and the surfaces
are close to a periodic array of balls (in fact, we give
explicit quantitative bounds on these facts).
\end{abstract}
\maketitle

\section*{Notation}

Most of the notation used in this paper is completely
standard. For the convenience of the reader, to avoid
ambiguities, we state it clearly
from the beginning.
The standard Euclidean basis of~$\R^n$ is denoted by~$e_1,\cdots,e_n$
(so that, in particular, $e_1=(1,0,\cdots,0)$).
If~$E\subseteq\R^n$ and~$v\in\R^n$, we use the notation
$$ E+v:=\{ p+v,\ p\in E\}.$$
Also,~$|E|$ is the Lebesgue measure of~$E$.
The $(n-1)$-dimensional Hausdorff measure is denoted by~${\mathcal{H}}^{n-1}$.

The notation~$\Delta$ will be used for the symmetric difference,
i.e.~$E\Delta F:= (E\setminus F)\cup (F\setminus E)$.
We denote by~$\chi_E$ the characteristic function of a set~$E$, i.e.
$$ \chi_E(x):=\left\{
\begin{matrix}
1 & {\mbox{ if }} x\in E,\\
0 & {\mbox{ if }} x\not\in E.
\end{matrix}
\right.$$
Also, throughout the paper, the world ``decreasing''
stands simply for ``non increasing''.

\section{Introduction}

The main goal of this paper is to
construct a nonlocal analogue of
the classical Delaunay surfaces (see~\cite{Del}), i.e.
surfaces that minimize a fractional perimeter functional among cylindrically
decreasing symmetric 
competitors that are periodic
in a given direction.
The notion of perimeter that we take into account
is a periodic functional of fractional type,
whose critical points are related to axially symmetric objects.

We also study the main geometric properties of the minimizers,
such as dislocation of mass and closeness to periodic array of balls.

For this scope, we will introduce a new fractional perimeter functional
that takes into account the periodicity of the surfaces
and we develop a fine analysis of the functional in order to
obtain suitable compactness properties.
The setting we work in is the following.
We consider a fractional parameter~$s\in(0,1)$. We
use coordinates $x=(x_1,x')\in\R\times\R^{n-1}=\R^n$, with~$n\ge2$,
and deal with the slab
$$S:=[-1/2,\,1/2]\times\R^{n-1}.$$
We consider the kernel~$K:\R^n\setminus(\Z\times\R^{n-1})\to\R$,
$$ K(x):=\sum_{k\in\Z}\frac{1}{|x+ke_1|^{n+s}}$$
and, given a set~$E\subseteq \R^n$, we define
$$ P_S(E):=\int_{E\cap S} \int_{S\setminus E} K(x-y)\,dx\,dy=
\int_{E\cap S} \int_{S\setminus E}
\sum_{k\in\Z}\frac{dx\,dy}{|x-y+ke_1|^{n+s}}.$$
This fractional functional is related to, but quite different from,
the nonlocal perimeter introduced in~\cite{CRS} (namely,
it shares with it some nonlocal features, but it has different scaling
behaviors and periodicity properties).
More precisely, 
on the one hand, the functional studied here may be considered
as a periodic version (in the horizontal direction) of
the fractional perimeter in~\cite{CRS}. On the other hand,
the kernel that we consider is non-standard, since it
has different scaling properties in the different coordinate directions.

We consider the class of our competitors\label{456scdfvg}
${\mathcal{K}}$, that is given by the sets~$F\subseteq S$
of the form
$$ F=\big\{ (x_1,x')\in S {\mbox{ with }} |x'|\le f(x_1)\big\},$$
for a given even function~$f:[-1/2,\,1/2]\to[0,+\infty]$ that is
decreasing in~$[0,\,1/2]$.

In this setting, we prove the existence
of volume constrained minimizers of~$P_S$ in~${\mathcal{K}}$:

\begin{theorem}\label{THE EX}
For any~$\mu>0$ there exists a minimizer for~$P_S$ in~${\mathcal{K}}$
with volume
constraint equal to~$\mu$. 

More explicitly, for any~$\mu>0$ there exists a set~$F_*\in{\mathcal{K}}$
such that~$|F_*|=\mu$ and, for any~$F\in{\mathcal{K}}$
such that~$|F|=\mu$, we have that~$P_S(F_*)\le P_S(F)$.
\end{theorem}

Recently, in the literature, there has been an intense
effort towards the construction of geometric object of nonlocal
nature that extend classical (i.e. local) ones, see e.g.~\cite{Dav, I5, I4}.
In some cases, the nonlocal objects inherit strong geometric
properties from the classical case, but also important
differences arise. 
In our setting, we think it is an interesting problem to determine whether
cylinders are minimizers for large volume.

As for small volumes,
the next result points out (in a quantitative way)
that in this case the minimizing
set does not put a considerable proportion of mass
close to the boundary of the slab (in particular, it is ``far from
being a cylinder''):

\begin{theorem}\label{TH DX}
Let
$$ F_*=
\Big\{ (x_1,x')\in S {\mbox{ s.t. }} |x'|\le f(x_1)\Big\}$$
be a minimizer with volume constraint~$\mu>0$, as given
in Theorem~\ref{THE EX}. Then
\begin{equation}\label{density 0}
\frac{f(1/4)}{ \mu^{\frac{1}{n-1}} }\le C\,\mu^{\frac{s}{n^2(n-1)}},\end{equation}
for some~$C>1$.
In particular, for any~$\delta\in(0,1)$, if~$\mu\in(0, C^{-1}\delta^{\frac{n^2}{s}})$,
we have that
\begin{equation}\label{density 1}
\frac{\big|F_*\cap \{|x_1|\ge 1/4\}\big|}{|F_*|}\le \delta
\end{equation}
for a suitable constant~$C>1$.
\end{theorem}

In case of small volumes, we also show that minimizers are close
to balls.
The notion of closeness will be measured by the so-called
Fraenkel asymmetry (or symmetric deficit) of a set~$E$, which
is defined as
$$ {\rm Def}(E):=\inf
\frac{|E\Delta B|}{|E|},$$
where the infimum is taken
over every ball $B\subset\R^n$ with $|B|=|E|$.
Roughly speaking, the Fraenkel asymmetry measures the $L^1$ distance
of~$E$ from being a ball of the same volume (the ball
may be conveniently translated in order to cover the set~$E$ as much
as possible, and the quantity above is normalized with
respect to the volume in order to be scale invariant).
In this setting we have:

\begin{theorem}\label{YES BALL}
Let~$F_*\subseteq S$ be a minimizer according to Theorem~\ref{THE EX},
with volume constraint~$\mu$. Then, if~$\mu$ is small enough,
$F_*$ is close to a ball. 

More precisely, for any~$\mu\in(0,1)$, we have that
$$ {\rm Def}(F_*)\le C\,
\mu^{\frac{n^2-s^2}{2 n^2(n-1)}}.$$
\end{theorem}

We observe that the control in the sense of deficit obtained in
Theorem~\ref{YES BALL}
may lead to further uniform (though perhaps less explicit)
asymptotic bounds, also
in the $C^1$-sense, by exploiting suitable density estimates
and approximation results, see e.g. Corollary~3.6 in~\cite{I5}
and Theorem 6.1 in~\cite{I4}.
\medskip

We point out that, after the present paper was completed
and submitted for publication, the very interesting
article~\cite{cabre} appeared, dealing
with surfaces of constant fractional mean curvature
(in the sense of~\cite{CRS, AV}).
This paper also
provides very fascinating fractional counterparts of
the Delaunay surfaces (in a different setting than the one
considered here,
and in the planar case, with an announcement of
the $n$-dimensional results to come).
\medskip

The rest of the paper is organized as follows. In Section~\ref{KD}
we study the decay properties of the kernel. Then, in Section~\ref{REL}
we give a detailed comparison between our functional and
the one in~\cite{CRS} (this is not only interesting
for seeing similarities and differences with the existing literature,
but it is also useful for constructing competitors and deriving
estimates). 

As a matter of fact, the proof of Theorem~\ref{THE EX}
also requires a careful energy analysis
and ad-hoc compactness arguments in order to
use the direct minimization method: these arguments are collected
in Sections~\ref{EB} and~\ref{CI2}.
With this, all the preliminary work will be completed, and we
will be able to prove
Theorems~\ref{THE EX}, \ref{TH DX}
and~\ref{YES BALL} in Sections~\ref{S1}, \ref{S4}
and~\ref{S5}, respectively.

The paper ends with two appendices. First, in Appendix~\ref{APP90A},
we show that the limit as~$s\nearrow1$ of our fractional functional
converges to the classical ``periodic'' perimeter 
(i.e. to the perimeter on the cylinder obtained by identifying the
``sides'' of the slab~$S$). Then, in Appendix~\ref{S rad},
we remark that the assumption of cylindrical symmetry for the competitors
in~${\mathcal{K}}$ can be relaxed (in the sense that our
fractional functional decreases under cylindrical rearrangements).

\section{Kernel decay}\label{KD}

First, we point out that our functional is compatible with the
periodic structure in the horizontal direction.
For this, if~$F\subseteq S$, we define
the periodic extension of~$F$ as
$$ F_{\rm per}:=\bigcup_{k\in\Z}(F+ke_1),$$
and we have:

\begin{lemma}\label{period}
For any~$\tau\in \R$, it holds that~$P_S
(F_{\rm per}+\tau e_1)=P_S(F_{\rm per})$.
\end{lemma}

\begin{proof} The function~$x_1\mapsto \chi_{ F_{\rm per} }(x_1,x')$
is $1$-periodic, and so is~$x_1\mapsto K(x_1,x')$, for any fixed~$x'\in\R^{n-1}$.
Therefore, for any~$\tau\in \R$,
\begin{eqnarray*}&& \int_{[-1/2,\,1/2]} \,dx_1 \int_{[-1/2,\,1/2]} \,dy_1
\,\chi_{ F_{\rm per} +\tau e_1}(x_1,x')
\chi_{ \R^n\setminus (F_{\rm per}+\tau e_1) }(y_1,y') K(x-y) 
\\ &=&
\int_{[-1/2,\,1/2]-\tau} \,dx_1 \int_{[-1/2,\,1/2]-\tau} \,dy_1
\,\chi_{ F_{\rm per} }(x_1,x')
\chi_{ \R^n\setminus F_{\rm per} }(y_1,y') K(x-y)
\\ &=&\int_{[-1/2,\,1/2]} \,dx_1 \int_{[-1/2,\,1/2]} \,dy_1
\,\chi_{ F_{\rm per} }(x_1,x')
\chi_{ \R^n\setminus F_{\rm per} }(y_1,y') K(x-y) .\end{eqnarray*}
Thus the desired result follows by integrating over~$x'$ and~$y'$.
\end{proof}

Now we prove a useful decay estimate on our kernel.
We remark that the scaling properties of our kernel
are quite different from the ones of many nonlocal problems
that have been studied in the literature: as a matter of
fact, the kernel that we study is not homogeneous
and it has quite different singular behaviors
locally and at infinity. Indeed, close to the origin
the dominant term is of the order of~$|x|^{-n-s}$,
but at infinity the $x_1$ direction ``averages out'',
as detailed in the following result:

\begin{lemma}\label{DEA}
For any~$x=(x_1,x')\in\R^n$ with~$|x'|\ge1$, we have
that~$K(x)\le C\,|x'|^{1-n-s}$,
for some~$C>0$.
\end{lemma}

\begin{proof} Fix~$x=(x_1,x')\in\R^n$ with~$|x'|\ge1$.
Let~$a_\pm :=\pm |x'|-x_1+1$ and~$b_\pm:=\pm3|x'|-x_1$. Let also~$y_\pm$ 
be the integer part of~$b_\pm$.
We observe that
$$ b_+=3|x'|-x_1 \ge |x'|+2-x_1=a_+ +1.$$
This says that at least one integer lies in the segment~$[a_+,b_+]$
and so~$y_+\ge a_+$. Therefore
$$ y_++x_1-1\ge a_+ +x_1-1 =|x'|.$$
Moreover
$$  y_-+x_1+1\le b_- +x_1+1\le -|x'|-1.$$
Consequently
\begin{eqnarray*}
 \sum_{{k\in\Z}\atop{k\in (-\infty,b_-]\cup[b_+,+\infty) }} \frac{1}{|x_1+k|^{n+s}}
&\le&
\sum_{{k\in\Z}\atop{k\ge y_+}} \frac{1}{(x_1+k)^{n+s}}
+
\sum_{{k\in\Z}\atop{k\le y_-+1}} \frac{1}{(-x_1-k)^{n+s}}
\\ &=&
\sum_{{k\in\Z}\atop{k\ge y_+}} \frac{1}{(x_1+k)^{n+s}}
+
\sum_{{k\in\Z}\atop{k\ge -y_- -1}} \frac{1}{(-x_1+k)^{n+s}}
\\ &\le&
\sum_{{k\in\Z}\atop{k\ge y_+}} \int_{x_1+k-1}^{x_1+k}\frac{dt}{t^{n+s}}
+
\sum_{{k\in\Z}\atop{k\ge -y_- -1}} \int_{-x_1+k-1}^{-x_1+k}
\frac{dt}{t^{n+s}}
\\ &=&
\int_{x_1+y_+-1}^{+\infty}\frac{dt}{t^{n+s}}
+
\int_{-x_1-y_--2}^{+\infty}\frac{dt}{t^{n+s}}
\\ &\le&
\int_{|x'|}^{+\infty}\frac{dt}{t^{n+s}}
+
\int_{|x'|}^{+\infty}\frac{dt}{t^{n+s}}
\\ &=&
\frac{C}{|x'|^{n+s-1}}.
\end{eqnarray*}
This says that
\begin{equation}\label{ES-1}
\sum_{{k\in\Z}\atop{k\not\in (-3|x'|-x_1,\,3|x'|-x_1)}} \frac{1}{|x_1+k|^{n+s}}
\le \frac{C}{|x'|^{n+s-1}}.
\end{equation}
Now, we observe that the interval~$[-3|x'|-x_1,\,3|x'|-x_1]$
has length~$6|x'|$ and so it contains at most~$6|x'|+1\le 7|x'|$
integers. This implies that
\begin{equation}\label{ES-3}
\sum_{{k\in\Z}\atop{ k\in [-3|x'|-x_1,\,3|x'|-x_1]}} \frac{1}{|x'|^{n+s}}
\le \frac{7}{|x'|^{n+s-1}}.\end{equation}
Moreover,
$$ |x+ke_1|^{n+s}
= \Big( |x_1+k|^2+|x'|^2\Big)^{\frac{n+s}{2}} \ge
\max \big\{ |x_1+k|^{n+s},\, |x'|^{n+s}\big\}.$$
Thus, recalling~\eqref{ES-1} and~\eqref{ES-3},
we conclude that
\begin{eqnarray*}
K(x)&\le& 
\sum_{{k\in\Z}\atop{{k\not\in (-3|x'|-x_1,\,3|x'|-x_1)}}}
\frac{1}{|x+ke_1|^{n+s}}
+\sum_{{k\in\Z}\atop{ k\in [-3|x'|-x_1,\,3|x'|-x_1]}} \frac{1}{|x+ke_1|^{n+s}}
\\ &\le&
\sum_{{k\in\Z}\atop{{k\not\in (-3|x'|-x_1,\,3|x'|-x_1)}}}
\frac{1}{|x_1+k|^{n+s}}
+\sum_{{k\in\Z}\atop{ k\in [-3|x'|-x_1,\,3|x'|-x_1]}} \frac{1}{|x'|^{n+s}}
\\ &\le&\frac{C+7}{|x'|^{n+s-1}},\end{eqnarray*}
which gives the desired claim up to renaming~$C$.
\end{proof}

\begin{corollary}\label{DEA-1}
Fix~$M\in\N$. We define
\begin{equation}\label{h M}
h_M(x):=\min\{M,\, K(x)\}.\end{equation}
Then, for any~$x\in\R^n$,
$$h_M(x)\le C_M \min\{1,\,|x'|^{1-n-s}\},$$
for some~$C_M>0$ possibly depending on~$M$.
\end{corollary}

\begin{proof} If~$|x'|\ge1$, we use Lemma~\ref{DEA} to see that
$$ \min\{1,\,|x'|^{1-n-s}\}=|x'|^{1-n-s} \ge C^{-1} K(x)
\ge C^{-1} h_M(x).$$
On the other hand, if~$|x'|<1$, we have that
$$ \min\{1,\,|x'|^{1-n-s}\}= 1\ge M^{-1} h_M(x).$$
Combining these two estimates we obtain the desired result.
\end{proof}

\section{Relation with the fractional perimeter}\label{REL}

The aim of this section is to point out the relation between
our functional and the fractional perimeter~${\rm Per}_s$
introduced in~\cite{CRS}. That is, we set
$$ {\rm Per}_s(F):=\int_{F} \int_{\R^n\setminus F}
\frac{dx\,dy}{|x-y|^{n+s}}$$
and we show that: on the one hand, our functional
is always below the fractional perimeter~${\rm Per}_s$,
on the other hand, our functional
is always above the fractional perimeter~${\rm Per}_s$,
up to a correction that depends on higher order volume terms,
and on a volume term coming from the boundary of the slab~$S$.
The precise statement goes as follows:

\begin{proposition}\label{EST PER-s}
Let~$F\subseteq S$. 
Then
\begin{equation}
\label{Per s 1} P_S(F)\le {\rm Per}_s(F).\end{equation}
More precisely, 
\begin{equation}\label{90}
{\rm Per}_s(F)-P_S(F)=
\sum_{{k\in\Z\setminus\{0\}}}
\int_{F} \int_{F}
\frac{dx\,dy}{|x-y+ke_1|^{n+s}}.\end{equation}
In addition,
if we set
\begin{equation}\label{hat} \begin{split}
&\widetilde F:= F\cap \{x_1\in [1/4,\,1/2]\},\\
&\widehat F:= (F+e_1)\cap \{x_1\in [1/2,\,3/4]\}\\
{\mbox{and }}\ & \Pi_S (F):= \int_{\widetilde F}\int_{\widehat F}
\frac{dx\,dy}{|x-y|^{n+s}},\end{split}\end{equation}
we have that
\begin{equation}\label{Per s 2} {\rm Per}_s(F)\le P_S(F)
+C\,\Big( |F|^2 + \Pi_S(F)\Big),\end{equation}
for some~$C>0$.
\end{proposition}

\begin{proof} We use the change of variable~$\widetilde y=y+ke_1$ to see that
$$ P_S(F)=
\int_{F} \int_{S\setminus F}
\sum_{k\in\Z}\frac{dx\,dy}{|x-y+ke_1|^{n+s}}
=
\int_{F} \int_{(S\setminus F)+ke_1}
\sum_{k\in\Z}\frac{dx\,d\widetilde y}{|x-\widetilde y|^{n+s}} 
=\int_{F} \int_{(S\setminus F)_{\rm per}}
\frac{dx\,d\widetilde y}{|x-\widetilde y|^{n+s}}.$$
We observe that~$(S\setminus F)_{\rm per}\subseteq\R^n\setminus F$,
so we obtain that
$$ P_S(F)\le \int_{F} \int_{(S\setminus F)_{\rm per}}
\frac{dx\,dy}{|x-y|^{n+s}}\le
\int_{F} \int_{\R^n\setminus F}
\frac{dx\,dy}{|x-y|^{n+s}}={\rm Per}_s(F).$$
This establishes~\eqref{Per s 1}. More generally, we see that
$$ (\R^n\setminus F)\,\setminus\, \big( (S\setminus F)_{\rm per}\big)
=\bigcup_{{k\in\Z\setminus\{0\}}} (F+ke_1),$$
therefore, with another change of variable, we have
\begin{equation}\label{ifuhc4rfofdsasdfg}
{\rm Per}_s(F)-P_S(F)=
\sum_{{k\in\Z\setminus\{0\}}}
\int_{F} \int_{F+ke_1}
\frac{dx\,dy}{|x-y|^{n+s}}=
\sum_{{k\in\Z\setminus\{0\}}}
\int_{F} \int_{F}
\frac{dx\,dy}{|x-y+ke_1|^{n+s}}.\end{equation}
This proves~\eqref{90}.
Now we observe that, if~$|k|\ge2$ and~$x$, $y\in S$, then
$$ |x-y+ke_1|\ge |x_1-y_1+k|\ge |k|-|x_1-y_1|\ge |k|-1\ge \frac{|k|}{2},$$
therefore
\begin{equation} \label{yuioghddds}
\sum_{{k\in\Z\setminus\{0\}}\atop{|k|\ge2}}
\int_{F} \int_{F}
\frac{dx\,dy}{|x-y+ke_1|^{n+s}}\le
\sum_{{k\in\Z\setminus\{0\}}\atop{|k|\ge2}}
\int_{F} \int_{F}
\frac{dx\,dy}{(|k|/2)^{n+s}}\le C\,|F|^2,\end{equation}
for some~$C>0$.
Moreover,
\begin{equation}\label{84} 
\int_{F} \int_{F+e_1} \chi_{[1/4,\,+\infty)}(|x_1-y_1|)
\frac{dx\,dy}{|x-y|^{n+s}} \le
\int_{F} \int_{F+e_1} 
\frac{dx\,dy}{(1/4)^{n+s}}\le C\,|F|^2,\end{equation}
for some~$C>0$.
Also, if~$x\in S$, $y\in S+e_1$ and~$|x_1-y_1|\le 1/4$, we have that
\begin{eqnarray*}&&
x_1\ge y_1 - |x_1-y_1|\ge -\frac{1}{2} + 1 -\frac14 =\frac14\\
{\mbox{and }}&& y_1-1\le |y_1-x_1|+x_1-1\le \frac14 +\frac12 -1 =-\frac14.
\end{eqnarray*}
As a consequence, if~$x\in F\subseteq S$ and~$y_1\in F+e_1\subseteq
S+e_1$, with~$|x-y|\le1/4$, we have that~$x\in \widetilde F$
and~$y\in \widehat F$, where the notation in~\eqref{hat} is here in use,
therefore
$$ \int_{F} \int_{F+e_1} \chi_{[0,\,1/4]}(|x_1-y_1|)
\frac{dx\,dy}{|x-y|^{n+s}} \le\Pi_S(F).
$$
This and~\eqref{84} give that
\begin{equation}\label{84bis}
\int_{F} \int_{F+e_1} 
\frac{dx\,dy}{|x-y|^{n+s}} \le
C\,\Big(|F|^2+\Pi_S(F)\Big).\end{equation}
Thus, using the change of variable~$\bar x:=x+e_1$
and~$\bar y:=y+e_1$, we also have that
\begin{equation}\label{85}
\int_{F} \int_{F-e_1} 
\frac{dx\,dy}{|x-y|^{n+s}} =
\int_{F+e_1} \int_{F}
\frac{dx\,dy}{|x-y|^{n+s}}\le
C\,\Big(|F|^2+\Pi_S(F)\Big).\end{equation}
Putting together~\eqref{84bis} and~\eqref{85} we obtain
$$ \sum_{{k\in\Z\setminus\{0\}}\atop{|k|\le 1}}
\int_{F} \int_{F}
\frac{dx\,dy}{|x-y+ke_1|^{n+s}}=
\int_{F} \int_{F+e_1}
\frac{dx\,dy}{|x-y|^{n+s}}+
\int_{F} \int_{F-e_1}
\frac{dx\,dy}{|x-y|^{n+s}}\le
C\,\Big(|F|^2+\Pi_S(F)\Big).$$
This and~\eqref{yuioghddds} imply that
$$ \sum_{{k\in\Z\setminus\{0\}}}
\int_{F} \int_{F}
\frac{dx\,dy}{|x-y+ke_1|^{n+s}}
\leq C\,\Big(|F|^2+\Pi_S(F)\Big).$$
Recalling~\eqref{ifuhc4rfofdsasdfg},
we see that this ends the proof of~\eqref{Per s 2}.
\end{proof}

As a consequence of our preliminary computations, we obtain that
cylinders have finite energy:

\begin{corollary}\label{finite}
The functional attains a finite value on cylinders.
Namely, for any~$R>0$, let~${\mathcal{C}} := \{(x_1,x')\in \R\times\R^{n-1}
{\mbox{ s.t. }}|x'|\le R\}$. Then~$P_S({\mathcal{C}}\cap S)<+\infty$.
\end{corollary}

\begin{proof} 
We take a set~$\widetilde{\mathcal{C}}$ with smooth boundary
and contained in~$[-1,\,1]\times\R^{n-1}$
such that~$\widetilde{\mathcal{C}}\cap S={\mathcal{C}}\cap S$.
Then, we use~\eqref{Per s 1}
and we obtain that
$$ +\infty > {\rm Per}_s (\widetilde{\mathcal{C}})\ge
P_S(\widetilde{\mathcal{C}}\cap S)=P_S({\mathcal{C}}\cap S),$$
as desired.
\end{proof}

Next result computes the term~$\Pi_S$ in~\eqref{Per s 2}
in the special case of small cylinders (this will play a role
in the proof of Theorem~\ref{YES BALL}).

\begin{lemma}\label{play}
For any~$r\in(0,\,1/4)$, let~${\mathcal{C}} := \{(x_1,x')\in S
{\mbox{ s.t. }}|x'|\le r\}$. Then~$\Pi_S({\mathcal{C}})\le C\,r^{n-s}$.
\end{lemma}

\begin{proof} We first translate in the first coordinate
and then change variable~$X:=x/r$ and~$Y:=y/r$, so that we obtain
\begin{equation}\label{87}\begin{split}
\Pi_S({\mathcal{C}})\,&=\int_{1/4}^{1/2}\,dx_1 \int_{1/2}^{3/4}\,dy_1
\int_{|x'|\le r} \,dx'\int_{|y'|\le r} \,dy' \frac{1}{|x-y|^{n+s}}
\\ &=\int_{-1/4}^{0}\,dx_1 \int_{0}^{1/4}\,dy_1
\int_{|x'|\le r} \,dx'\int_{|y'|\le r} \,dy' \frac{1}{|x-y|^{n+s}}
\\ &= r^{n-s}
\int_{-1/(4r)}^{0}\,dX_1 \int_{0}^{1/(4r)}\,dY_1
\int_{|X'|\le 1} \,dX'\int_{|Y'|\le 1} \,dY' \frac{1}{|X-Y|^{n+s}}
.\end{split}\end{equation}
Now we take a
bounded set~${\mathcal{C}}^\star\subset [-1,0]\times\R^{n-1}$
with smooth boundary that contains~$\{(x_1,x')\in\R^n {\mbox{ s.t. }} 
x_1\in [-1,0] {\mbox{ and }} |x'|\le 1\}$. Then we have that
\begin{equation*}
\int_{-1}^{0}\,dX_1 \int_{0}^{1}\,dY_1
\int_{|X'|\le 1} \,dX'\int_{|Y'|\le 1} \,dY' \frac{1}{|X-Y|^{n+s}}
\le \int_{{\mathcal{C}}^\star}\int_{\R^n\setminus {\mathcal{C}}^\star}
\frac{dX\,dY}{|X-Y|^{n+s}} \le {\rm Per}_s ({\mathcal{C}}^\star)\le C,
\end{equation*}
for some~$C>0$, thus~\eqref{87} becomes
\begin{equation}\label{88}\begin{split}
\Pi_S({\mathcal{C}})
\le C\, r^{n-s} & \left( 1 + 
\int_{-1/(4r)}^{-1}\,dX_1 \int_{0}^{1/(4r)}\,dY_1
\int_{|X'|\le 1} \,dX'\int_{|Y'|\le 1} \,dY' \frac{1}{|X-Y|^{n+s}}
\right.\\&\left. 
+\int_{-1/(4r)}^{0}\,dX_1 \int_{1}^{1/(4r)}\,dY_1
\int_{|X'|\le 1} \,dX'\int_{|Y'|\le 1} \,dY' \frac{1}{|X-Y|^{n+s}}
\right).\end{split}\end{equation}
Notice that we can change variable~$(x,y):=(-Y,-X)$ 
and see that
\begin{eqnarray*}&& \int_{-1/(4r)}^{0}\,dX_1 \int_{1}^{1/(4r)}\,dY_1
\int_{|X'|\le 1} \,dX'\int_{|Y'|\le 1} \,dY' \frac{1}{|X-Y|^{n+s}}
\\ &&\qquad=
\int_0^{1/(4r)}\,dy_1 \int_{-1/(4r)}^{-1}\,dx_1
\int_{|y'|\le 1} \,dy'\int_{|x'|\le 1} \,dx' \frac{1}{|x-y|^{n+s}}.\end{eqnarray*}
As a consequence, we can write~\eqref{88} as
\begin{equation}\label{89}
\Pi_S({\mathcal{C}})
\le C\, r^{n-s}\left( 1+2           
\int_{-1/(4r)}^{-1}\,dx_1 \int_{0}^{1/(4r)}\,dy_1
\int_{|x'|\le 1} \,dx'\int_{|y'|\le 1} \,dy' \frac{1}{|x-y|^{n+s}}
\right).\end{equation}
Now we observe that
\begin{eqnarray*}
&& \int_{-1/(4r)}^{-1}\,dx_1 \int_{0}^{1/(4r)}\,dy_1
\int_{|x'|\le 1} \,dx'\int_{|y'|\le 1} \,dy' \frac{1}{|x-y|^{n+s}} \\
&\le& 
\int_{-1/(4r)}^{-1}\,dx_1 \int_{0}^{1/(4r)}\,dy_1
\int_{|x'|\le 1} \,dx'\int_{|y'|\le 1} \,dy' \frac{1}{|x_1-y_1|^{n+s}}\\
&=&
C \,\int_{-1/(4r)}^{-1}\,dx_1 \int_{0}^{1/(4r)}\,dy_1
\frac{1}{|x_1-y_1|^{n+s}} \\
&\le& C \,\int_{-1/(4r)}^{-1}\,dx_1 (-x_1)^{1-n-s} \\
&\le& C\, \int_{1}^{+\infty}\,d\tau \,\tau^{1-n-s}
\\ &\le& C.
\end{eqnarray*}
The desired result thus follows by plugging this estimate into~\eqref{89}.
\end{proof}

\section{Energy bounds}\label{EB}

We consider here an auxiliary energy functional
and we prove that the functional~$P_S$ is bounded from below
by it. The proof requires a very careful analysis of
the different contributions and the result,
together with the one in the subsequent Proposition~\ref{converge},
will play a crucial role for the proof of Theorem~\ref{THE EX},
since it will lead to the compactness of the minimizing sequences.

\begin{proposition}\label{energy}
Let~$F\subseteq S$. Suppose that there exists
an even 
function~$f:[-1/2,\,1/2]\rightarrow [0,+\infty]$,
with $f$ decreasing in~$[0,\,1/2]$,
such that
$$ F:= \Big\{ (x_1,x')\in S {\mbox{ s.t. }} |x'|\le f(x_1)\Big\}.$$
Let~$\eps_*\le \alpha_*\in [0,\,1/2]$ and suppose that
\begin{equation}\label{eps star}
f(x_1)\ge 4 {\mbox{ for all }} x_1\in [0,\eps_*)
\end{equation}
and
\begin{equation}\label{a star}
f(x_1)\ge 2{f(y_1)} {\mbox{ for all }} x_1\in [0,\eps_*)
{\mbox{ and }} y_1\in (\alpha_*,\,1/2]. 
\end{equation}
Then
$$ P_S(F)\ge C \,
\int_{0}^{\eps_*} \left[
\int_{\alpha_*}^{1/2} 
\frac{f^{n-1}(x_1)}{|x_1-y_1|^{1+s}}\,dy_1\right]\,dx_1,$$
for a suitable constant~$C>0$.
\end{proposition}

\begin{proof} We have
\begin{eqnarray*}
P_S(F)&=&\int_{F} \int_{S\setminus F} K(x-y)\,dx\,dy \\
&\ge&\int_{F} \int_{S\setminus F} \frac{dx\,dy}{|x-y|^{n+s}}
\\ &=& \int_{-1/2}^{1/2} dx_1\,
\int_{-1/2}^{1/2} dy_1 \int_{\{ |x'|\le f(x_1)\}} dx'
\int_{\{ |y'|> f(y_1)\}} dy' \frac{1}{|x-y|^{n+s}}. 
\end{eqnarray*}
Now we introduce cylindrical coordinates by writing~$x'= \rho\theta$
and~$y'= r\omega$, with~$\theta$, $\omega\in S^{n-2}$. We obtain
\begin{equation}\label{uisgfdsajhgfd3456} P_S(F) \geq
\int_{-1/2}^{1/2} dx_1\,
\int_{-1/2}^{1/2} dy_1 \int_{S^{n-2}}\,d\theta
\int_{S^{n-2}}\,d\omega
\int_0^{f(x_1)} d\rho
\int_{f(y_1)}^{+\infty} dr \frac{ \rho^{n-2} r^{n-2} }{
\big( |x_1-y_1|^2 +|\rho\theta-r\omega|^2\big)^{\frac{n+s}{2}} }
.\end{equation}
Here and in the sequel, $d\theta$ and~$d\omega$
are short notations for~$d{\mathcal{H}}^{n-1}(\theta)$
and~$d{\mathcal{H}}^{n-1}(\omega)$, respectively.
Now, for any~$\theta\in S^{n-2}$, we consider a rotation~$R_\theta$
on~$S^{n-2}$
such that~$\theta =R_\theta e_2$.
Then, we can rotate~$\omega=R_\theta\widetilde\omega$,
and obtain that
\begin{eqnarray*}
&& \int_{S^{n-2}}\,d\omega \frac{ \rho^{n-2} r^{n-2} }{
\big( |x_1-y_1|^2 +|\rho\theta-r\omega|^2\big)^{\frac{n+s}{2}} }=
\int_{S^{n-2}}\,d\omega \frac{ \rho^{n-2} r^{n-2} }{
\big( |x_1-y_1|^2 +|\rho R_\theta e_2 -r\omega|^2\big)^{\frac{n+s}{2}} }\\
&&\qquad=
\int_{S^{n-2}}\,d\widetilde\omega \frac{ \rho^{n-2} r^{n-2} }{
\big( |x_1-y_1|^2 +|R_\theta(\rho e_2 -r\widetilde\omega)|^2\big)^{\frac{n+s}{2}} }
=\int_{S^{n-2}}\,d\widetilde\omega \frac{ \rho^{n-2} r^{n-2} }{
\big( |x_1-y_1|^2 +|\rho e_2 -r\widetilde\omega|^2\big)^{\frac{n+s}{2}} }
.\end{eqnarray*}
Notice that the latter integral in now independent of~$\theta$.
Then we can write~\eqref{uisgfdsajhgfd3456} as
\begin{equation}\label{uisgfdsajhgfd3456-2} P_S(F) \geq C
\int_{-1/2}^{1/2} dx_1\,
\int_{-1/2}^{1/2} dy_1 
\int_{S^{n-2}}\,d\omega
\int_0^{f(x_1)} d\rho
\int_{f(y_1)}^{+\infty} dr \frac{ \rho^{n-2} r^{n-2} }{
\big( |x_1-y_1|^2 +|\rho e_2-r\omega|^2\big)^{\frac{n+s}{2}} }
,\end{equation}
for some~$C>0$. 

Now we observe that
$$ |\rho e_2-r\omega|^2 = \rho^2 +r^2 -2\rho r \omega_2
= |\rho-r|^2 +2\rho r(1-\omega_2).$$
where~$\omega_2=\omega\cdot e_2$ is the second component of the vector~$\omega
\in S^{n-2}\subset \{0\}\times\R^{n-1}$.
Now\footnote{For concreteness we suppose in this part that~$n\geq3$.
In the very special case~$n=2$,
one does not have any component~$(\omega_3,\cdots,\omega_n)$,
so she or he can just disregard~$\underline\omega$ and go directly to~\eqref{uisgfdsajhgfd3456-3}.} 
we define~$\underline\omega:=(\omega_3,\cdots,\omega_n)$.
Then~$\omega=(0,\omega_2,\underline\omega)$
and
\begin{equation}\label{67d8cv9bbfrff}
\omega_2^2+|\underline \omega|^2=1. 
\end{equation}
We also set
$$ S^{n-2}_\star:=
\left\{ \omega =(0,\omega_2,\underline\omega)\in S^{n-2}
{\mbox{ s.t. }} \omega_2\ge \frac{9}{10}\right\}.$$
Using~\eqref{67d8cv9bbfrff}, we see that,
$$ \left\{ |\underline\omega|\le \frac{\sqrt{19}}{10}\right\}
\subseteq S^{n-2}_\star.$$
Also, in~$S^{n-2}_\star$,
$$ 2\rho r (1-\omega_2)=
\frac{2\rho r (1-\omega_2^2)}{1+\omega_2} \le 2\rho r (1-\omega_2^2)=
2\rho r|\underline\omega|^2.$$
Therefore, fixed any~$a\ge0$, we have
\begin{eqnarray*}
&& \int_{S^{n-2}} 
\frac{ d\omega }{
\big( a^2 +|\rho e_2-r\omega|^2\big)^{\frac{n+s}{2}} }
=
\int_{S^{n-2}} 
\frac{ d\omega }{
\big( a^2 + |\rho-r|^2 +2\rho r(1-\omega_2)\big)^{\frac{n+s}{2}} }\\
&&\qquad \ge \int_{S^{n-2}_\star}
\frac{ d\omega }{
\big( a^2 + |\rho-r|^2 +2\rho r(1-\omega_2)\big)^{\frac{n+s}{2}} }\ge
\int_{S^{n-2}_\star}
\frac{ d\omega }{
\big( a^2 + |\rho-r|^2 +2\rho r |\underline\omega|^2\big)^{\frac{n+s}{2}} }
\\ &&\qquad \ge C\,
\int_{\{|\underline\omega|\le \sqrt{19}/10\}}
\frac{ d\underline\omega }{
\big( a^2 + |\rho-r|^2 +2\rho r |\underline\omega|^2\big)^{\frac{n+s}{2}} },
\end{eqnarray*}
for some~$C>0$, possibly different from line to line.
So we use polar coordinates~$\R^{n-2}\ni \underline\omega = R \varphi$,
with~$\varphi\in S^{n-3}$ and obtain
from the latter estimate that
\begin{equation} \label{jkcvbndssfghrewefg}
\int_{S^{n-2}}
\frac{ d\omega }{
\big( a^2 +|\rho e_2-r\omega|^2\big)^{\frac{n+s}{2}} }
\ge
C\,
\int_0^{\sqrt{19}/10}
\frac{ R^{n-3}\, dR}{
\big( a^2 + |\rho-r|^2 +2\rho r R^2\big)^{\frac{n+s}{2}} }.\end{equation}
Now we observe that, for any~$X$, $Y\ge0$, we have that
\begin{equation} \label{789dfvgjhgfdbvcxw}
\int_0^{\sqrt{19}/10}
\frac{ R^{n-3}\, dR}{
\big( X^2+Y^2 R^2\big)^{\frac{n+s}{2}} }
=\left(\frac{X}{Y}\right)^{n-2}\, \frac{1}{X^{n+s}}\,
\int_0^{ \frac{ \sqrt{19}\,Y }{10\,X}}
\frac{ t^{n-3}\, dt}{
\big( 1+t^2\big)^{\frac{n+s}{2}} },\end{equation}
where the change of variable~$R=Xt/Y$ was performed.

Now we denote, for any~$x\ge0$,
$$ \Gamma(x):=
\int_0^{x}
\frac{ t^{n-3}\, dt}{
\big( 1+t^2\big)^{\frac{n+s}{2}} }.$$
We observe that if~$x\in[0,1]$ and~$t\in[0,x]$, then~$1+t^2\le 2$ and so
$$ \Gamma(x)\ge C\,\int_0^x t^{n-3}\,dt = C\,x^{n-2},$$
up to renaming~$C>0$.
Moreover, if~$x\ge1$, then
$$ \Gamma(x)\ge
\int_0^{1}
\frac{ t^{n-3}\, dt}{
\big( 1+t^2\big)^{\frac{n+s}{2}} }\ge C.$$
Summarizing, for any~$x\ge0$, we have that
$$ \Gamma(x)\ge C\,\min\{ x^{n-2},\, 1\}.$$
Thus, going back to~\eqref{789dfvgjhgfdbvcxw},
\begin{equation} \label{78009dfvgjhgfdbvcxw-78}
\begin{split} &\int_0^{\sqrt{19}/10}
\frac{ R^{n-3}\, dR}{
\big( X^2+Y^2 R^2\big)^{\frac{n+s}{2}} }
=\left(\frac{X}{Y}\right)^{n-2}\, \frac{1}{X^{n+s}}\,
\Gamma \left(\frac{\sqrt{19}\,Y}{10\,X}\right)
\\ &\qquad \ge C\, \left(\frac{X}{Y}\right)^{n-2}\, \frac{1}{X^{n+s}}\,
\min\left\{ 1,\, \left(\frac{Y}{X}\right)^{n-2}\right\}
= C\, \frac{1}{X^{n+s}}\,
\min\left\{ \left(\frac{X}{Y}\right)^{n-2},\, 1\right\}
.\end{split}\end{equation}
Now we take~$X:=\sqrt{a^2+|\rho-r|^2}$ and~$Y:=\sqrt{2\rho r}$
and we plug~\eqref{78009dfvgjhgfdbvcxw-78} into~\eqref{jkcvbndssfghrewefg}.
In this way we obtain
\begin{equation} \label{000jkcvbndssfghrewefg}
\int_{S^{n-2}}
\frac{ d\omega }{
\big( a^2 +|\rho e_2-r\omega|^2\big)^{\frac{n+s}{2}} }
\ge
C\, \frac{1}{\big( a^2+|\rho-r|^2\big)^{\frac{n+s}{2}}}\,
\min\left\{ 
\left(\frac{a^2+|\rho-r|^2}{\rho r}\right)^{\frac{n-2}{2}},\, 1\right\}
.\end{equation}
Hence we take~$a:=|x_1-y_1|$ and we insert~\eqref{000jkcvbndssfghrewefg}
into~\eqref{uisgfdsajhgfd3456-2}, obtaining that
\begin{equation}\label{uisgfdsajhgfd3456-3} \begin{split}
P_S(F) \,&\ge C\,
\int_{-1/2}^{1/2} dx_1\,
\int_{-1/2}^{1/2} dy_1
\int_0^{f(x_1)} d\rho
\int_{f(y_1)}^{+\infty} dr \\
&\qquad\quad\frac{ \rho^{n-2} r^{n-2} }{
\big( |x_1-y_1|^2+|\rho-r|^2\big)^{\frac{n+s}{2}} }\,
\min\left\{
\left(\frac{|x_1-y_1|^2+|\rho-r|^2}{\rho r}\right)^{\frac{n-2}{2}},\, 1\right\}
.\end{split}\end{equation}
Now we observe that
\begin{eqnarray*}
&& \int_{-1/2}^{1/2} dx_1\,\int_{-1/2}^{1/2} dy_1
\int_0^{f(x_1)} d\rho
\int_{f(y_1)}^{+\infty} dr \\
&&\qquad\chi_{\{ |x_1-y_1|^2+|\rho-r|^2\le \rho r\}}\,
\frac{ \rho^{n-2} r^{n-2} }{
\big( |x_1-y_1|^2+|\rho-r|^2\big)^{\frac{n+s}{2}} }\,
\min\left\{
\left(\frac{|x_1-y_1|^2+|\rho-r|^2}{\rho r}\right)^{\frac{n-2}{2}},\, 1\right\}\\
&=&
\int_{-1/2}^{1/2} dx_1\,
\int_{-1/2}^{1/2} dy_1
\int_0^{f(x_1)} d\rho
\int_{f(y_1)}^{+\infty} dr 
\, \chi_{ \{ |x_1-y_1|^2+|\rho-r|^2\le \rho r\} }\,
\frac{ \rho^{\frac{n-2}{2}} r^{\frac{n-2}{2}} }{
\big( |x_1-y_1|^2+|\rho-r|^2\big)^{1+\frac{s}{2}} }
.\end{eqnarray*}
This and~\eqref{uisgfdsajhgfd3456-3} give that
\begin{equation}\label{uisgfdsajhgfd3456-45} 
P_S(F) \ge C\,
\int_{-1/2}^{1/2} dx_1\,
\int_{-1/2}^{1/2} dy_1
\int_0^{f(x_1)} d\rho
\int_{f(y_1)}^{+\infty} dr
\, \chi_{ \{ |x_1-y_1|^2+|\rho-r|^2\le \rho r\} }\,
\frac{ \rho^{\frac{n-2}{2}} r^{\frac{n-2}{2}} }{
\big( |x_1-y_1|^2+|\rho-r|^2\big)^{1+\frac{s}{2}} }.
\end{equation}
Accordingly, we perform the change of variable~$\rho=|x_1-y_1|\,\alpha$
and~$r=|x_1-y_1|\,\beta$, so that~\eqref{uisgfdsajhgfd3456-45}
becomes
\begin{eqnarray*} P_S(F) &\ge &C\,
\int_{-1/2}^{1/2} dx_1\,
\int_{-1/2}^{1/2} dy_1
\int_0^{\frac{f(x_1)}{|x_1-y_1|}} d\alpha
\int_{\frac{f(y_1)}{|x_1-y_1|}}^{+\infty} d\beta
\, \chi_{ \{ 1+|\alpha-\beta|^2\le \alpha\beta\} }\,
\frac{ |x_1-y_1|^{n-2-s}\,\alpha^{\frac{n-2}{2}} \beta^{\frac{n-2}{2}} }{
\big( 1+|\alpha-\beta|^2\big)^{1+\frac{s}{2}} }
\\ &\ge& C\,
\int_{0}^{\eps_*} dx_1\,
\int_{\alpha_*}^{1/2} dy_1
\int_{\frac{9 f(x_1)}{10 \,|x_1-y_1|}}^{\frac{f(x_1)}{|x_1-y_1|}} d\alpha
\int_{\frac{f(y_1)}{|x_1-y_1|}}^{+\infty} d\beta
\, \chi_{ \{ 1+|\alpha-\beta|^2\le \alpha\beta\} }\,
\frac{ |x_1-y_1|^{n-2-s}\,\alpha^{\frac{n-2}{2}} \beta^{\frac{n-2}{2}} }{
\big( 1+|\alpha-\beta|^2\big)^{1+\frac{s}{2}} } ,
\end{eqnarray*}
where~$\eps_*$ and~$\alpha_*$
were introduced in~\eqref{eps star}
and~\eqref{a star}.
As a matter of fact, using~\eqref{a star}, we obtain that,
in the domain above,
$$ \frac{f(y_1)}{|x_1-y_1|}\le
\frac{f(x_1)}{2\,|x_1-y_1|} < \frac{9 f(x_1)}{10\,|x_1-y_1|}\le \alpha.$$
As a consequence
\begin{equation}\label{yuidfghvnmnbvcf} \begin{split} P_S(F)\, &\ge C\,
\int_{0}^{\eps_*} dx_1\,
\int_{\alpha_*}^{1/2} dy_1
\int_{\frac{9 f(x_1)}{10 \,|x_1-y_1|}}^{\frac{f(x_1)}{|x_1-y_1|}} d\alpha
\int_{\alpha}^{+\infty} d\beta
\, \chi_{ \{ 1+|\alpha-\beta|^2\le \alpha\beta\} }\,
\frac{ |x_1-y_1|^{n-2-s}\,\alpha^{\frac{n-2}{2}} \beta^{\frac{n-2}{2}} }{
\big( 1+|\alpha-\beta|^2\big)^{1+\frac{s}{2}} }\\
&\ge C\,
\int_{0}^{\eps_*} dx_1\,
\int_{\alpha_*}^{1/2} dy_1
\int_{\frac{9 f(x_1)}{10 \,|x_1-y_1|}}^{\frac{f(x_1)}{|x_1-y_1|}} d\alpha
\int_{\alpha}^{\alpha+1} d\beta
\, \chi_{ \{ 1+|\alpha-\beta|^2\le \alpha\beta\} }\,
\frac{ |x_1-y_1|^{n-2-s}\,\alpha^{\frac{n-2}{2}} \beta^{\frac{n-2}{2}} }{
\big( 1+|\alpha-\beta|^2\big)^{1+\frac{s}{2}} } .
\end{split}\end{equation}
Now we observe that
in the domain above~$0\le\beta-\alpha\le1$, therefore, recalling~\eqref{eps star},
$$ 1+|\alpha-\beta|^2\le 2 \le \frac{f^2(x_1)}{8} \le
\frac{f^2(x_1)}{8\,|x_1-y_1|^2} $$
and
\begin{equation}\label{mngfd45678} \alpha\beta \ge \alpha^2 \ge
\frac{81\,f^2(x_1)}{100\,|x_1-y_1|^2} >\frac{f^2(x_1)}{8\,|x_1-y_1|^2},\end{equation}
that is
$$ 1+|\alpha-\beta|^2 < \alpha\beta.$$
Accordingly, \eqref{yuidfghvnmnbvcf} boils down to
$$ P_S(F)\ge C\,
\int_{0}^{\eps_*} dx_1\,
\int_{\alpha_*}^{1/2} dy_1
\int_{\frac{9 f(x_1)}{10 \,|x_1-y_1|}}^{\frac{f(x_1)}{|x_1-y_1|}} d\alpha
\int_{\alpha}^{\alpha+1} d\beta
\; \frac{ |x_1-y_1|^{n-2-s}\,\alpha^{\frac{n-2}{2}} \beta^{\frac{n-2}{2}} }{
\big( 1+|\alpha-\beta|^2\big)^{1+\frac{s}{2}} }
.$$
Hence, using again~\eqref{mngfd45678},
$$ P_S(F)
\ge C\,
\int_{0}^{\eps_*} dx_1\,
\int_{\alpha_*}^{1/2} dy_1
\int_{\frac{9 f(x_1)}{10 \,|x_1-y_1|}}^{\frac{f(x_1)}{|x_1-y_1|}} d\alpha
\int_{\alpha}^{\alpha+1} d\beta
\; \frac{ |x_1-y_1|^{n-2-s}}{
\big( 1+|\alpha-\beta|^2\big)^{1+\frac{s}{2}} }\cdot
\frac{f^{n-2}(x_1)}{|x_1-y_1|^{n-2}}.$$
Now we point out that
$$ \int_\alpha^{\alpha+1}
\frac{ d\beta}{
\big( 1+|\alpha-\beta|^2\big)^{1+\frac{s}{2}} }=
\int_0^{1}       
\frac{ d\tau}{
\big( 1+\tau^2\big)^{1+\frac{s}{2}} }\ge C,$$
and thus we get that
\begin{eqnarray*}  P_S(F)
&\ge& C\,
\int_{0}^{\eps_*} dx_1\,
\int_{\alpha_*}^{1/2} dy_1
\int_{\frac{9 f(x_1)}{10 \,|x_1-y_1|}}^{\frac{f(x_1)}{|x_1-y_1|}} d\alpha
\; |x_1-y_1|^{n-2-s}
\cdot\frac{f^{n-2}(x_1)}{|x_1-y_1|^{n-2}}\\
&\ge&
C\,
\int_{0}^{\eps_*} dx_1\,
\int_{\alpha_*}^{1/2} dy_1
\; |x_1-y_1|^{n-2-s}
\cdot\frac{f^{n-2}(x_1)}{|x_1-y_1|^{n-2}}\cdot
\frac{f(x_1)}{|x_1-y_1|},\end{eqnarray*}
which completes the proof of Proposition~\ref{energy}.
\end{proof}

\section{Convergence issues}\label{CI2}

Here we show that uniform energy bounds, as the ones obtained
in Proposition~\ref{energy}, joined with volume constraints,
imply a suitable compactness.

\begin{proposition}\label{converge}
Let~$C_1$, $C_2>0$.
Consider a sequence of 
even functions~$f_k:[-1/2,\,1/2]\rightarrow [0,+\infty]$,
with $f_k$ decreasing in~$[0,\,1/2]$.
Assume that
for any~$\eps_{*,k}\le \alpha_{*,k}\in [0,\,1/2]$ such that
\begin{equation}\label{LAL1}
f_k(x_1)\ge 4 {\mbox{ for all }} x_1\in [0,\eps_{*,k})
\end{equation}
and
\begin{equation}\label{LAL2}
f_k(x_1)\ge 2{f_k(y_1)} {\mbox{ for all }} x_1\in [0,\eps_{*,k})
{\mbox{ and }} y_1\in (\alpha_{*,k},\,1/2]
\end{equation}
it holds that
\begin{equation}\label{LAL3}
\int_{0}^{\eps_{*,k}} \left[
\int_{\alpha_{*,k}}^{1/2}
\frac{f_k^{n-1}(x_1)}{|x_1-y_1|^{1+s}}\,dy_1\right]\,dx_1\le C_1.
\end{equation}
Suppose also that
\begin{equation}\label{5.3bis} \int_{0}^{1/2} f_k^{n-1}(x_1)\,dx_1 = C_2.\end{equation}
Then, there exists a function~$f$ such that,
up to a subsequence, $f_k\to f$ a.e. in $[-1/2,\,1/2]$ as~$k\to+\infty$, and
$$ \int_{0}^{1/2} f^{n-1}(x_1)\,dx_1 = C_2.$$
\end{proposition}

\begin{proof} First
we point out that, for any~$r\in[0,1/2]$,
$$ C_2 \ge \int_{0}^{r} f^{n-1}_k(x_1)\,dx_1
\ge r\,\inf_{ [0,\,r] } f_k^{n-1},$$
that is
\begin{equation}\label{inf f}
f_k(r)=\inf_{ [0,\,r] } f_k\le\left(\frac{C_2}{r}\right)^{\frac{1}{n-1}}.
\end{equation}
In particular~$f_k(1/2)\le (2C_2)^{\frac{1}{n-1}}$.
So, for any~$M>4+(2C_2)^{\frac{1}{n-1}}$, we are allowed to define
\begin{equation}\label{5.4bis}
\eps_k(M):= \inf \{ r \in[0,\,1/2] {\mbox{ s.t. }}
f_k(r)<M\}.\end{equation}
We claim that
\begin{equation}\label{inf f eps}
\eps_k(M)\le \frac{C_2}{M^{n-1}}.
\end{equation}
Indeed, if~$\eps_k(M)=0$ we are done, so we assume~$\eps_k(M)>0$.
Then, for any~$k$ and~$M$ fixed,
for any~$j\in\N$, we can take~$r_j\in[(1-2^{-j})\,\eps_k(M),\,\eps_k(M)]$
such that~$f_k(r_j)\ge M$. Hence,
from~\eqref{inf f}, we have that
$$ M\le f_k(r_j)\le \left(\frac{C_2}{r_j}\right)^{\frac{1}{n-1}}.$$
So we pass~$j\to+\infty$ and we obtain
$$ M\le \left(\frac{C_2}{\eps_k(M)}\right)^{\frac{1}{n-1}},$$
that proves~\eqref{inf f eps}.

Now we define~$\alpha_k(M):=\eps_k(M/2)$.
We point out that, if~$x_1\in [0,\eps_{k}(M))$ then
$f_k(x_1)\ge M$, and therefore~$f_k(x_1)\ge 4$. Also,
if~$y_1\in (\alpha_{k}(M),\,1/2]=(\eps_{k}(M/2),\,1/2]$, we have that~$
f_k(y_1)\le M/2$. Accordingly,
if~$x_1\in [0,\eps_{k}(M))$ and~$y_1\in (\alpha_{k}(M),\,1/2]$,
we have that~$f(x_1)\ge M\ge 2f(y_1)$.

These considerations show that~\eqref{LAL1} and~\eqref{LAL2}
are satisfied by~$\eps_{*,k}:=\eps_k(M)$ and~$\alpha_{*,k}:=
\alpha_k(M)$. As a consequence
\begin{equation}\label{LAL4}
{\mbox{ formula \eqref{LAL3} holds true with
$\eps_{*,k}:=\eps_k(M)$ and $\alpha_{*,k}:=
\alpha_k(M)$.}}
\end{equation}
Now we claim that there exists a constant~$C_\star>1$ such that,
if~$M\ge C_\star$ then
\begin{equation}\label{FUNF INT}
\int_0^{\eps_k(M)} f_k^{n-1}(x_1)\,dx_1\le \frac{C_\star}{M^{s(n-1)}}.
\end{equation}
To prove this, we use~\eqref{LAL4} to notice that 
\begin{equation}\label{id8ocugdsqwww} {C_1}\ge
\int_{0}^{\eps_k(M)} \left[ \int_{\alpha_k(M)}^{1/2}
f_k^{n-1}(x_1)\,\frac{1}{|x_1-y_1|^{1+s}}\,dy_1\right]\,dx_1.\end{equation}
Now, fixed~$x_1<\eps_k(M)\le \eps_k(M/2)=\alpha_k(M)$, we compute
\begin{equation}\label{7dfisq1df}
\int_{\alpha_k(M)}^{1/2} \frac{1}{|x_1-y_1|^{1+s}}\,dy_1
=\int_{\alpha_k(M)}^{1/2} \frac{1}{(y_1-x_1)^{1+s}}\,dy_1
= \frac{1}{s} \left[\frac{1}{(\alpha_k(M)-x_1)^s}
-\frac{1}{((1/2)-x_1)^s}
\right].\end{equation}
Now, if~$x_1<\eps_k(M)$, we have that
$$ \frac12 -x_1 \ge \frac12 -\eps_k(M)\ge \frac12 -
\frac{C_2}{M^{n-1}} \ge\frac14,$$
if~$M$ is sufficiently large (independently on~$k$), thanks to~\eqref{inf f eps}.
Moreover, using again~\eqref{inf f eps}, we see that
$$ \alpha_k(M)-x_1\le \alpha_k(M)=\eps_k(M/2)\le
\frac{C_3}{M^{n-1}},$$
for some~$C_3>0$.
Therefore~\eqref{7dfisq1df}
implies that
\begin{equation*}
\int_{\alpha_k(M)}^{1/2} \frac{1}{|x_1-y_1|^{1+s}}\,dy_1
\ge \left[ {C_4}\,{M^{s(n-1)}}
-C_5
\right]\ge C_6 M^{s(n-1)},\end{equation*}
for suitable constants~$C_i>0$, as long as~$M$ is large
enough, independently on~$k$.

So we plug this information into~\eqref{id8ocugdsqwww}, and we conclude that
$$ C_7\ge M^{s(n-1)}
\int_{0}^{\eps_k(M)} f_k^{n-1}(x_1)\,dx_1,$$
and this proves~\eqref{FUNF INT}.

Now we claim that
\begin{equation}\label{c.a.e.}
f_k\to f {\mbox{ a.e. in }}[0,1/2],{\mbox{ as }}k\to+\infty,
\end{equation}
up to subsequences, for some function~$f$. To prove this,
we use the compactness of the functions with bounded variation,
joined with a diagonal trick.
We fix~$M\in\N$, $M\ge1$, and we use \eqref{inf f eps}
to see that~$(\eps_k(M),\,1/2)\supseteq (C_2M^{1-n},\,1/2)$,
for any~$k$,
and so, by~\eqref{5.4bis},
$$ \sup_{ (C_2M^{1-n},\,1/2) } f_k\le \sup_{(\eps_k(M),\,1/2)}f_k\le M.$$
Since~$f_k$ is monotone, this gives that~$\| f_k\|_{BV(C_2M^{1-n},\,1/2)}\le 2M$.
As a consequence of this
and of the
compactness of bounded variations functions
(see e.g. Theorem~3.23 in~\cite{AFP}) we get that, for any fixed~$M\in\N$,
$f_k\to f^{(M)}$ a.e. in~$(C_2M^{1-n},\,1/2)$, up to a subsequence,
for some function~$f^{(M)} :(C_2M^{1-n},\,1/2)\to[0,+\infty]$.
More explicitly, we write this subsequence by introducing
an increasing function~$\phi_M:\N\to\N$, and by stating that
\begin{equation}\label{5.13}
f_{\phi_M\circ\dots\circ\phi_1(k)}\to f^{(M)} {\mbox{ a.e. in }} (C_2M^{1-n},\,1/2).
\end{equation}
As a matter of fact, for a.e. $x_1\in (C_2M^{1-n},\,1/2)$, we have that
$$ f^{(M+1)}(x_1)=\lim_{k\to+\infty} f_{\phi_{M+1}\circ
\phi_M\circ\dots\circ\phi_1(k)}(x_1)=
\lim_{k\to+\infty} f_{\phi_M\circ\dots\circ\phi_1(k)}(x_1)=f^{(M)}(x_1),$$
so we can define~$f:(0,1/2)\to[0,+\infty]$ by setting~$f(x_1):=
f^{(M)}(x_1)$, for some~$M$ so large that~$C_2M^{1-n}<x_1$.
Hence, we consider the diagonal subsequence
$f_{\phi_k\circ\dots\circ\phi_1(k)}$ and we prove that it converges to~$f$
a.e. in~$(0,\,1/2)$. For this scope, we fix~$\eps>0$,
$x_1\in (0,\,1/2)$
(possibly outside a set of measure zero) and~$M_{x_1}
\in\N$ such that~$C_2M_{x_1}^{1-n}<x_1$ and
we use~\eqref{5.13}
to find~$k(\eps,x_1)$ such that, if~$k\ge k(\eps,x_1)$, then
$$ |f_{\phi_{M_{x_1}}\circ\dots\circ\phi_1(k)} (x_1)-f(x_1)|\le \eps.$$
Now, if~$k\ge M_{x_1}+k(\eps,x_1)$,
we have that~$\phi_k\circ\dots\circ\phi_1(k)$
is a subsequence of~$\phi_{M_{x_1}}\circ\dots\circ\phi_1(k)$ and thus
$$ |f_{\phi_{k}\circ\dots\circ\phi_1(k)} (x_1)-f(x_1)|\le \eps.$$
Since~$\eps$ is arbitrary, this shows that~$f_{\phi_k\circ\dots\circ\phi_1(k)}(x_1)\to
f(x_1)$, which in turn completes the proof of~\eqref{c.a.e.}.

Now we identify $f_k$ with the subsequence constructed
in~\eqref{c.a.e.} and prove that
\begin{equation}\label{c.L1}
f_k^{n-1}\to f^{n-1} {\mbox{ in }}L^1(0,1/2),{\mbox{ as }}k\to+\infty.
\end{equation}
For this scope, we fix~$\delta>0$
and we use~\eqref{c.a.e.} and Egoroff's Theorem to find~$E_\delta\subseteq(0,1/2)$
such that~$|(0,1/2)\setminus E_\delta|\le\delta$ and~$f_k^{n-1}\to f^{n-1}$
uniformly on~$E_\delta$. We choose~$M_\delta:= \delta^{-1/2}$
and we use~\eqref{FUNF INT} to conclude that
\begin{equation*}
\int_0^{\eps_k(M_\delta)} f_k^{n-1}(x_1)\,dx_1\le
\frac{C_\star}{M_\delta^{s(n-1)}}= C_\star \delta^{\frac{s(n-1)}{2}},\end{equation*}
for every~$k$. 
Therefore
\begin{equation}\label{POI1}\begin{split}
& \int_0^{1/2} |f^{n-1}_k(x_1)-f^{n-1}_h(x_1)|\,dx_1\le
\int_{(0,1/2)\setminus E_\delta} \big(
f^{n-1}_k(x_1)+f^{n-1}_h(x_1)\big)\,dx_1
+ \|f^{n-1}_k-f_h^{n-1}\|_{L^\infty(E_\delta)} |E_\delta|\\
&\quad\le 2C_\star \delta^{\frac{s(n-1)}{2}}+
\int_{(\eps_k(M_\delta),1/2)\setminus E_\delta} f^{n-1}_k(x_1)\,dx_1
+\int_{(\eps_h(M_\delta),1/2)\setminus E_\delta} f^{n-1}_h(x_1)\,dx_1
+\|f^{n-1}_k-f_h^{n-1}\|_{L^\infty(E_\delta)}
.\end{split}\end{equation}
Now we recall that, by~\eqref{5.4bis}, $f_k\le M_\delta$
in~$(\eps_k(M_\delta),1/2)$, thus
$$ \int_{(\eps_k(M_\delta),1/2)\setminus E_\delta} f^{n-1}_k(x_1)\,dx_1
\le M_\delta\, \big|(\eps_k(M_\delta),1/2)\setminus E_\delta\big|
\le M_\delta \delta=\sqrt\delta.$$
The same holds with~$h$ instead of~$k$. Consequently, formula~\eqref{POI1}
gives that
$$ \int_0^{1/2} |f^{n-1}_k(x_1)-f^{n-1}_h(x_1)|\,dx_1\le
2C_\star \delta^{\frac{s(n-1)}{2}} +2\sqrt\delta+
\|f^{n-1}_k-f_h^{n-1}\|_{L^\infty(E_\delta)}.$$
So, if we choose~$h$, $k$ so large that~$\|f^{n-1}_k-f_h^{n-1}\|_{L^\infty(E_\delta)}\le\sqrt\delta$,
we obtain
$$ \int_0^{1/2} |f^{n-1}_k(x_1)-f^{n-1}_h(x_1)|\,dx_1\le
2C_\star \delta^{\frac{s(n-1)}{2}} +3\sqrt\delta.$$
Since $\delta$ was arbitrarily fixed, we have just shown that~$f_k$
is a Cauchy sequence on~$L^1(0,1/2)$, which implies~\eqref{c.L1}.

The desired claim now follows from~\eqref{c.L1} and an even reflection
in~$(-1/2,0)$.
\end{proof}

\section{Proof of Theorem~\ref{THE EX}}\label{S1}

The proof uses the direct methods of the calculus of variations,
combined with
the fine estimates of Propositions~\ref{energy}
and~\ref{converge}. That is, we take a minimizing sequence
of sets~$F_k\in{\mathcal{K}}$ with~$|F_k|=\mu$ and
\begin{equation}\label{sequence} 
\lim_{k\to+\infty} P_S(F_k) =\inf_{{F\subseteq S}\atop{|F|=\mu}} P_S(F).\end{equation}
Since~$F_k\in{\mathcal{K}}$, 
we have that~$F_k$
has the form
$$ F_k= \Big\{ (x_1,x')\in S {\mbox{ s.t. }} |x'|\le f_k(x_1)\Big\},$$
with~$f_k:[-1/2,\,1/2]\to[0,+\infty]$ even and decreasing in~$[0,1/2]$.
We remark that, for any~$k\in\N$,
$$ \mu = |F_k| = C_o \,\int_0^{1/2} f_k^{n-1}(x_1)\,dx_1,$$
for some dimensional constant~$C_o>0$.
Also, we can fix a set~$F_o$ with~$P_S(F_o)<+\infty$
(recall for instance Corollary~\ref{finite}),
and we may assume that
\begin{equation}\label{6.1bis}
{\mbox{$P_S(F_k)\le
P_S(F_o)$ for every~$k\in\N$.}}
\end{equation}
We want to prove that
\begin{equation}\label{-CONV-1}
{\mbox{up to a subsequence, $f_k\to f$ a.e. in $[-1/2,\,1/2]$,
with }} C_o\,\int_0^{1/2} f^{n-1}(x_1)\,dx_1=\mu.
\end{equation}
Indeed, if there is a sequence along which
$$ \sup_{k\in\N} \|f_k\|_{L^\infty([-1/2,\,1/2])} <+\infty$$
then
$$ \sup_{k\in\N} \|f_k^{n-1}\|_{BV([-1/2,\,1/2])} <+\infty$$
and so, again up to a subsequence,
we can pass to the limit in $L^1([-1/2,\,1/2])$
and a.e. in~$[-1/2,\,1/2]$, see e.g. Theorem~3.23 in~\cite{AFP},
and obtain~\eqref{-CONV-1}. Thus, we can suppose that
$$ \sup_{k\in\N} \|f_k\|_{L^\infty([-1/2,\,1/2])} =+\infty.$$
In this case, we check that the assumptions of Proposition~\ref{converge}
are satisfied. For this, let~$\eps_{*,k}\le \alpha_{*,k}\in [0,\,1/2]$ such that
\begin{equation}\label{LAL1bis}
f_k(x_1)\ge 4 {\mbox{ for all }} x_1\in [0,\eps_{*,k})
\end{equation}
and
\begin{equation}\label{LAL2bis}
f_k(x_1)\ge 2{f_k(y_1)} {\mbox{ for all }} x_1\in [0,\eps_{*,k})
{\mbox{ and }} y_1\in (\alpha_{*,k},\,1/2].\end{equation}
We observe that~\eqref{LAL1bis} and~\eqref{LAL2bis}
say that~\eqref{eps star}
and~\eqref{a star} are satisfied (for all the indices~$k$),
hence we can use
Proposition~\ref{energy} and conclude that, for every~$k\in\N$,
$$ P_S(F_k)\ge C \,
\int_{0}^{\eps_{*,k}} \left[
\int_{\alpha_{*,k}}^{1/2}
\frac{f^{n-1}_k(x_1)}{|x_1-y_1|^{1+s}}\,dy_1\right]\,dx_1,$$
for a suitable~$C>0$.
Therefore, for every~$k\in\N$,
$$ P_S(F_o)\ge C \,
\int_{0}^{\eps_{*,k}} \left[
\int_{\alpha_{*,k}}^{1/2}
\frac{f^{n-1}_k(x_1)}{|x_1-y_1|^{1+s}}\,dy_1\right]\,dx_1,$$
thanks to~\eqref{6.1bis},
and this gives that condition~\eqref{LAL3}
is satisfied in this case. Consequently, 
\eqref{-CONV-1} follows from Proposition~\ref{converge}.

Thus, we define
$$ F_* :=\Big\{ (x_1,x')\in S {\mbox{ s.t. }} |x'|\le f(x_1)\Big\},$$
and we show that~$F_*$ is the desired minimizer.
First of all, $|F_*|=\mu$, thanks to the integral constraint in~\eqref{-CONV-1}.
Furthermore, as~$k\to+\infty$,
\begin{equation}\label{sequence2}
{\mbox{$\chi_{F_k}\to\chi_{F_*}$ a.e. in $S$.}}
\end{equation}
To check this, we recall that, by~\eqref{-CONV-1},
$f_k\to f$ in~$S\setminus Z_1$, with~$|Z_1|=0$.
Moreover, we have that, for any fixed~$x_1\in[-1/2,\,1/2]$,
the set~$\Lambda^{x_1}:=
\{x'\in\R^{n-1}
{\mbox{ s.t. }} |x'|=f(x_1)\}$ is a sphere in~$\R^{n-1}$
and so it is of measure zero in~$\R^{n-1}$ (in symbols, $|\Lambda^{x_1}|=0$).
Thus, if~$Z_2:= \{(x_1,x')\in S {\mbox{ s.t. }}|x'|=f(x_1) \}$,
we have, by Fubini's Theorem, that
$$ |Z_2| =\int_{-1/2}^{1/2}\,dx_1 \int_{\R^{n-1}}\,dx'\chi_{\{f(x_1)=|x'|\}}(|x'|)
=\int_{-1/2}^{1/2}\,dx_1 |\Lambda^{x_1}|=0.$$
Therefore, \eqref{sequence2} would follow if we show that~$f_k\to f$
in~$S\setminus (Z_1\cup Z_2)$. For this, fix~$x\in S\setminus (Z_1\cup Z_2)$.
Since~$x\not\in Z_2$ we have that either~$|x'|<f(x_1)$
or~$|x'|>f(x_1)$. Since~$x\not\in Z_1$ we have that~$f_k(x)\to f(x)$,
so that for large~$k$, either~$|x'|<f_k(x_1)$
or~$|x'|>f_k(x_1)$, respectively. This shows that, for large~$k$,
$x\in F_*$ if and only if~$x\in F_k$, therefore~$\chi_{F_k}(x)=\chi_{F}(x)$,
and this proves~\eqref{sequence2}.

Consequently, using~\eqref{sequence}
and~\eqref{sequence2}, we have, by Fatou Lemma,
\begin{eqnarray*}
&&\inf_{{F\subseteq S}\atop{|F|=\mu}} P_S(F)=
\lim_{k\to+\infty} P_S(F_k) =
\lim_{k\to+\infty}
\int_{F_k} \int_{S\setminus F_k} K(x-y)\,dx\,dy \\ &&\qquad=
\lim_{k\to+\infty}
\int_{S} \int_{S} \chi_{F_k}(x)\,\chi_{S\setminus F_k}(y)\,K(x-y)\,dx\,dy
\ge \int_{S} \int_{S} \chi_{F_*}(x)\,\chi_{S\setminus F_*}(y)\,K(x-y)\,dx\,dy
=P_S(F_*).
\end{eqnarray*}
This shows the desired minimization property and it ends the proof
of Theorem~\ref{THE EX}.

\section{Proof of Theorem~\ref{TH DX}}\label{S4}

We suppose that~$f(1/4)\ge \beta \mu^{\frac{1}{n-1}}$, for some~$\beta\ge0$,
and we obtain an estimate on~$\beta$. For this we use the volume constraint
and we observe that
$$ C_o \mu =\int_{0}^{1} f^{n-1}(x_1)\,dx_1\ge
\int_{0}^{1/8} f^{n-1}(x_1)\,dx_1\ge \frac{f^{n-1}(1/8)}{8},$$
for some~$C_o>0$.
That is, by monotonicity, we have that
$$ \beta \mu^{\frac{1}{n-1}} \le f(1/4) \le f(x_1)\le f(1/8)\le 
C_1\mu^{\frac{1}{n-1}}$$
for every~$x_1\in[1/8,\,1/4]$ (here and in the sequel~$C_i>0$
is an appropriate dimensional constant). Notice that this already says that
\begin{equation}\label{beta}
\beta\le C_1.
\end{equation}
As a consequence, we see that~\eqref{density 0}
is obvious if~$\mu\ge 1/16^{n-1}$, so we suppose from now on
that
\begin{equation}\label{16}\mu\in(0,1/16^{n-1}).\end{equation}
We let~$B$ the ball with volume~$\mu$
(say, centered at the origin) and we
use the minimality property of~$F_*$ and~\eqref{Per s 1}
to see that
\begin{eqnarray*}
C_2\mu^{\frac{n-s}{n}} &=& {\rm Per}_s(B)
\\ &\ge& P_S(B)\\
&\ge& P_S(F_*) \\
&\ge& C_3 \int_{1/8}^{1/4} \,dx_1
\int_{1/8}^{1/4} \,dy_1 \int_{|x'|\le f(x_1)}\,dx'
\int_{|y'|>f(y_1)}\,dy' \frac{1}{|x-y|^{n+s}} \\
&\ge& C_3 \int_{1/8}^{1/4} \,dx_1
\int_{1/8}^{1/4} \,dy_1 \int_{|x'|\le \beta \mu^{\frac{1}{n-1}}}\,dx'
\int_{|y'|> C_1 \mu^{\frac{1}{n-1}}}\,dy' \frac{1}{|x-y|^{n+s}} .
\end{eqnarray*}
Now we let~$M:=\mu^{-\frac{1}{n-1}}$. Notice that~$M\ge 16$,
thanks to~\eqref{16},
thus, if~$N$ is the integer part of~$M/8$, we have that~$N\le M/8$
and
\begin{equation}\label{8.2bis} N\ge \frac{M}{8}-1\ge\frac{M}{16}.\end{equation}
Hence we change variables~$X:= Mx$ and~$Y:=My$ 
and then we translate in the first coordinate, and we obtain
\begin{equation}\label{83}\begin{split}
C_4\mu^{\frac{n-s}{n}}\,&\ge M^{s-n}
\int_{M/8}^{M/4} \,dX_1
\int_{M/8}^{M/4} \,dY_1 \int_{|X'|\le \beta M \mu^{\frac{1}{n-1}}}\,dX'
\int_{|Y'|> C_1 M \mu^{\frac{1}{n-1}}}\,dY' \frac{1}{|X-Y|^{n+s}}
\\ &=
M^{s-n}
\int_{0}^{M/8} \,dx_1
\int_{0}^{M/8} \,dy_1 \int_{|x'|\le \beta }\,dx'
\int_{|y'|> C_1}\,dy' \frac{1}{|x-y|^{n+s}}
\\ &\ge
M^{s-n}
\int_{0}^{N} \,dx_1
\int_{0}^{N} \,dy_1 \int_{|x'|\le \beta }\,dx'
\int_{|y'|> C_1}\,dy' \frac{1}{|x-y|^{n+s}}
\\ &\ge
M^{s-n}\sum_{k=1}^N
\int_{k-1}^{k} \,dx_1
\int_{k-1}^{k} \,dy_1 \int_{|x'|\le \beta }\,dx'
\int_{|y'|> C_1}\,dy' \frac{1}{|x-y|^{n+s}}
\\ &=
M^{s-n}\sum_{k=1}^N
\int_{0}^{1} \,dx_1
\int_{0}^{1} \,dy_1 \int_{|x'|\le \beta }\,dx'
\int_{|y'|> C_1}\,dy' \frac{1}{|x-y|^{n+s}}
\\ &= M^{s-n} N \int_{0}^{1} \,dx_1
\int_{0}^{1} \,dy_1 \int_{|x'|\le \beta }\,dx'
\int_{|y'|> C_1}\,dy' \frac{1}{|x-y|^{n+s}}
.\end{split}\end{equation}
Now we observe that 
if~$x_1$, $y_1\in[0,\,1]$,
$|x'|\le\beta$ and~$|y'|> C_1+1$, then
$$ |x_1-y_1|\le 1 \le C_1+1-\beta\le |y'|-|x'|\le |x'-y'|,$$
thanks to~\eqref{beta}, thus in this case
$$ |x-y|\le C_5 |x'-y'|\le C_5\big( |x'|+|y'|\big)\le
C_5 \big( \beta+|y'|\big)\le C_5 \big( C_1+|y'|\big)\le 2C_5\,|y'|.$$
Accordingly
\begin{eqnarray*} && \int_{0}^{1} \,dx_1
\int_{0}^{1} \,dy_1 \int_{|x'|\le \beta }\,dx'
\int_{|y'|> C_1}\,dy' \frac{1}{|x-y|^{n+s}} \ge
\int_{0}^{1} \,dx_1
\int_{0}^{1} \,dy_1 \int_{|x'|\le \beta }\,dx'
\int_{|y'|> C_1+1}\,dy' \frac{1}{|x-y|^{n+s}} \\
&&\quad \ge C_6 \int_{0}^{1} \,dx_1
\int_{0}^{1} \,dy_1 \int_{|x'|\le \beta }\,dx'
\int_{|y'|> C_1+1}\,dy' \frac{1}{|y'|^{n+s}} = C_7 \beta^{n-1}.
\end{eqnarray*}
By inserting this information into~\eqref{83} and using~\eqref{8.2bis}
we obtain that
$$ \mu^{\frac{n-s}{n}}\ge C_8 \,M^{s-n}N\beta^{n-1}\ge \frac{C_8\,M^{1+s-n}\beta^{n-1}}{16}=
C_9 \mu^{\frac{n-s-1}{n-1}}\beta^{n-1},$$
and this gives that~$\beta\le C_{10} \mu^{\frac{s}{n^2(n-1)}}$.
This proves~\eqref{density 0}.

Now we prove~\eqref{density 1}. For this we use the monotonicity of~$f$,
the volume constraint
and~\eqref{density 0} to compute
$$\frac{\big|F_*\cap \{|x_1|\ge 1/4\}\big|}{|F_*|}= C_{11}\mu^{-1}
\int_{1/4}^1 f^{n-1}(x_1)\,dx_1 \le C_{12} \mu^{-1} f^{n-1}(1/4)
\le C_{13} \mu^{\frac{s}{n^2}}, $$
which implies~\eqref{density 1}.
The proof of Theorem~\ref{TH DX}
is thus complete.

\section{Proof of Theorem~\ref{YES BALL}}\label{S5}

We take $B$ the ball of volume~$\mu$ (say, centered at the origin).
Using the minimality of~$F_*$ and Proposition~\ref{EST PER-s},
we see that
\begin{equation}\label{678d9f033ehgfsdn}
0\le \frac{P_S(B)-P_S(F_*)}{\mu^{\frac{n-s}{n}}}
\le\frac{ {\rm Per}_s(B)
-{\rm Per}_s(F_*)
+C\,\Big( |F_*|^2 + \Pi_S(F_*)\Big) }{\mu^{\frac{n-s}{n}}}.
\end{equation}
Now we use Theorem~\ref{TH DX}, so we write
$$ F_*=
\Big\{ (x_1,x')\in S {\mbox{ s.t. }} |x'|\le f(x_1)\Big\}$$
with
$$ \frac{f(1/4)}{ \mu^{\frac{1}{n-1}} }\le C\,\mu^{\frac{s}{n^2(n-1)}}.$$
In particular, using the monotonicity of~$f$, we have that, for any~$x_1\in
[1/4,\,1/2]$
$$ f(x_1)\le f(1/4) \le C\,\mu^{\frac{s}{n^2(n-1)}} \times \mu^{\frac{1}{n-1}}
=:r.$$
This (in the notation of~\eqref{hat}) says that~$\widetilde F_*$
and~$\widehat F_*$ are contained in the cylinder of radius~$r$, and therefore,
by Lemma~\ref{play},
$$ \Pi_S(F_*)\le C\,r^{n-s} = C \mu^{\frac{s(n-s)}{n^2(n-1)}} 
\times \mu^{\frac{n-s}{n-1}}.$$
Using this and the fact that~$|F_*|=\mu=|B|$, we write~\eqref{678d9f033ehgfsdn}
as
\begin{equation}\label{1-11678d9f033ehgfsdn}
0\le\frac{ {\rm Per}_s(B) }{|B|^{\frac{n-s}{n}}}
-\frac{ {\rm Per}_s(F_*) }{|F_*|^{\frac{n-s}{n}}}
+C\,\Big( \mu^{\frac{n+s}{n}}
+ 
\mu^{\frac{n^2-s^2}{n^2(n-1)}} 
\Big) .
\end{equation}
Now we use the quantitative isoperimetric inequality in
Theorem~1.1 of~\cite{I5}, according to which
$$ \frac{ {\rm Per}_s(F_*) }{|F_*|^{\frac{n-s}{n}}}
\ge
\frac{ {\rm Per}_s(B) }{|B|^{\frac{n-s}{n}}}
\,\Big( 1+ C\,{\rm Def}^2(F^*)\Big).$$
By inserting this into~\eqref{1-11678d9f033ehgfsdn}
we conclude that
\begin{equation}\label{9sicx3ed} 0\le -\frac{ {\rm Per}_s(B) }{|B|^{\frac{n-s}{n}}}\,
C\,{\rm Def}^2(F^*)
+C\,\Big( \mu^{\frac{n+s}{n}}
+
\mu^{\frac{n^2-s^2}{n^2(n-1)}}\Big).\end{equation}
Since
$$ \frac{n^2-s^2}{n^2(n-1)} =
\frac{n-s}{n (n-1)} 
\frac{(n+s)}{n} \le 
\frac{n}{n (n-1)}
\frac{(n+s)}{n} \le \frac{(n+s)}{n},$$
we see that~\eqref{9sicx3ed}
implies the thesis of Theorem~\ref{YES BALL}.

\section*{Acknowledgements} 
J. D{\'a}vila and M. del Pino
have been supported by grants Fondecyt 1130360 and 1150066, Fondo
Basal
CMM and by Millenium Nucleus CAPDE NC130017.
S. Dipierro has been supported by Alexander von Humboldt-Stiftung.
E. Valdinoci
has been supported by PRIN grant
201274FYK7 ``Critical Point Theory
and Perturbative Methods for Nonlinear Differential Equations''
and ERC
grant 277749 ``EPSILON Elliptic Pde's and Symmetry of Interfaces and
Layers for Odd Nonlinearities''.

\appendix

\section{Limit of $P_S$ as~$s\nearrow1$}\label{APP90A}

Now we show that (a suitably scaled version of) our nonlocal
perimeter functional~$P_S$ approaches the classical perimeter as~$s\nearrow1$.
Notice that of course the functional~$P_S$ depends on the
fractional parameter~$s\in(0,1)$ (though we did not
keep track explicitly on this dependence when it was not
necessary to use it). Also, heuristically, points ``close
to each other'', up to periodicity, provide the biggest contribution to~$P_S$,
due to the singularity of the kernel. A rigorous version
of this concept is given by the following result:

\begin{proposition}
Let~$F\in {\mathcal{K}}$ be a set with~$(\partial F)\cap \{ |x_1|<1/2\}$
of class~$C^2$. Then
$$ \lim_{s\nearrow1} \frac{1-s}{\omega_{n-1}}\,P_S(F)= {\mathcal{H}}^{n-1}
\big((\partial F)\cap \{ |x_1|<1/2\}\big).$$
\end{proposition}

\begin{proof} First of all, we fix~$\lambda\in(0,1/4)$,
to be taken as small as we wish in the sequel, and we define
$$ F_\lambda:= F\cap \left\{ |x_1|\in \left[ \frac{1}{2}-\lambda,\;
\frac{1}{2}\right] \right\}.$$
We observe that if~$x\in F$ and~$y\in F\setminus F_\lambda$
then
$$ |x-y\pm e_1|\ge |x_1-y_1\pm 1|\ge |1 \mp y_1| -|x_1|
\ge \frac{1}{2}+\lambda -|x_1|\ge \lambda.$$
Similarly, if~$y\in F_\lambda$ and~$x\in F\setminus F_\lambda$
then
$$ |x-y\pm e_1|\ge |x_1-y_1\pm 1|\ge |1 \pm x_1| -|y_1|
\ge \frac{1}{2}+\lambda -|y_1|\ge \lambda.$$
Consequently,
\begin{equation*}
\sum_{k\in\{-1,+1\}} \left[
\int_{F}\int_{F\setminus F_\lambda} \frac{dx\,dy}{|x-y+ke_1|^{n+s}}
+
\int_{F\setminus F_\lambda}\int_{F_\lambda}
\frac{dx\,dy}{|x-y+ke_1|^{n+s}} \right] 
\le 4\, |F|^2\,\lambda^{-n-s}.\end{equation*}
As a consequence
\begin{equation}\label{9ucdiwedf}
\lim_{s\nearrow1}(1-s)
\sum_{k\in\{-1,+1\}} \left[
\int_{F}\int_{F\setminus F_\lambda} \frac{dx\,dy}{|x-y+ke_1|^{n+s}}
+
\int_{F\setminus F_\lambda}\int_{F_\lambda}
\frac{dx\,dy}{|x-y+ke_1|^{n+s}} \right]=0,
\end{equation}
for any fixed~$\lambda\in(0,1/4)$.

Now we claim that
\begin{equation}\label{9cidfvdhvb}
\lim_{\lambda\searrow0}\lim_{s\nearrow1}
(1-s) \sum_{k\in\{-1,+1\}} \int_{F_\lambda}\int_{F_\lambda}
\frac{dx\,dy}{|x-y+ke_1|^{n+s}} =\omega_{n-1}\left[
{\mathcal{H}}^{n-1} (\partial F)-
{\mathcal{H}}^{n-1} \big((\partial F)\cap\{|x_1|<1/2\}\big)
\right].\end{equation}
To prove this, we write
$$ F_\lambda^+:= F\cap \left\{ x_1\in \left[ \frac{1}{2}-\lambda,\;
\frac{1}{2}\right] \right\} \;{\mbox{ and }}\;
F_\lambda^-:= F\cap \left\{ x_1\in \left[ -\frac{1}{2},\;
-\frac{1}{2}+\lambda
\frac{1}{2}\right] \right\}.$$
Notice that
\begin{equation}\label{0asdcvf444}
F_\lambda=F_\lambda^+\cup F_\lambda^-\end{equation}
and that
if~$x\in F_\lambda^+$ and~$y\in F_\lambda^+ +e_1$,
or if~$x\in F_\lambda^-$ and~$y\in F_\lambda^+ +e_1$,
or if~$x\in F_\lambda^-$ and~$y\in F_\lambda^- +e_1$,
then we have that~$|x-y|\ge 1/4$, and so
$$ \int_{F_\lambda^+}\int_{F_\lambda^+ +e_1}
\frac{dx\,dy}{|x-y|^{n+s}} +
\int_{F_\lambda^-}\int_{F_\lambda^+ +e_1}
\frac{dx\,dy}{|x-y|^{n+s}} 
+\int_{F_\lambda^-}\int_{F_\lambda^- +e_1}
\frac{dx\,dy}{|x-y|^{n+s}} \le 3\cdot 4^{n+s} |F|^2$$
and therefore
\begin{equation}\label{45678dvfgee-1}
\lim_{s\nearrow1}
(1-s)\left[ 
\int_{F_\lambda^+}\int_{F_\lambda^+ +e_1}
\frac{dx\,dy}{|x-y|^{n+s}} +
\int_{F_\lambda^-}\int_{F_\lambda^+ +e_1}
\frac{dx\,dy}{|x-y|^{n+s}}
+\int_{F_\lambda^-}\int_{F_\lambda^- +e_1}
\frac{dx\,dy}{|x-y|^{n+s}} 
\right]=0.
\end{equation}
Now, we define~$\Omega_\lambda$
as the interior of the closure of~$F_\lambda^+\cup(F_\lambda^- +e_1)$.
By Lemma~11 in~\cite{CV},
we have that
$$ \lim_{s\nearrow1}
(1-s)\int_{F_\lambda\cap\Omega_\lambda}
\int_{\Omega_\lambda\setminus F_\lambda}
\frac{dx\,dy}{|x-y|^{n+s}} = \omega_{n-1}{\rm Per} (F_\lambda,\Omega_\lambda)
=\omega_{n-1} {\mathcal{H}}^{n-1} \big(F\cap\{x_1=1/2\}\big).$$
Accordingly,
\begin{equation}\label{45678dvfgee-2}\begin{split} &
\lim_{\lambda\searrow0}\lim_{s\nearrow1}
(1-s)\int_{F_\lambda^+}
\int_{F_\lambda^-+e_1}
\frac{dx\,dy}{|x-y|^{n+s}}=
\lim_{\lambda\searrow0}\lim_{s\nearrow1}
(1-s)\int_{F_\lambda\cap\Omega_\lambda}
\int_{\Omega_\lambda\setminus F_\lambda}
\frac{dx\,dy}{|x-y|^{n+s}}\\&\qquad = \omega_{n-1}
{\mathcal{H}}^{n-1} \big(F\cap\{x_1=1/2\}\big).\end{split}
\end{equation}
Now, we decompose the set~$F_\lambda$ as in~\eqref{0asdcvf444}
and we change variable, to see that
\begin{eqnarray*}
&& \int_{F_\lambda}\int_{F_\lambda}
\frac{dx\,dy}{|x-y-e_1|^{n+s}} =
\int_{F_\lambda}\int_{F_\lambda+e_1}
\frac{dx\,dy}{|x-y|^{n+s}} \\
&&\quad=
\int_{F_\lambda^-}\int_{F_\lambda^-+e_1}
\frac{dx\,dy}{|x-y|^{n+s}}+
\int_{F_\lambda^-}\int_{F_\lambda^++e_1}
\frac{dx\,dy}{|x-y|^{n+s}}+
\int_{F_\lambda^+}\int_{F_\lambda^++e_1}
\frac{dx\,dy}{|x-y|^{n+s}}+
\int_{F_\lambda^+}\int_{F_\lambda^-+e_1}
\frac{dx\,dy}{|x-y|^{n+s}}.
\end{eqnarray*}
Using this, \eqref{45678dvfgee-1}
and~\eqref{45678dvfgee-2}, we obtain that
\begin{equation}\label{0sde4hhhhajs-1}
\lim_{\lambda\searrow0}\lim_{s\nearrow1}
(1-s)
\int_{F_\lambda}\int_{F_\lambda}
\frac{dx\,dy}{|x-y-e_1|^{n+s}}=
\omega_{n-1}{\mathcal{H}}^{n-1} \big(F\cap\{x_1=1/2\}\big).
\end{equation}
Similarly, one sees that
\begin{equation}\label{0sde4hhhhajs-2}
\lim_{\lambda\searrow0}\lim_{s\nearrow1}
(1-s)
\int_{F_\lambda}\int_{F_\lambda}
\frac{dx\,dy}{|x-y+e_1|^{n+s}}=
\omega_{n-1}{\mathcal{H}}^{n-1} \big(F\cap\{x_1=-1/2\}\big).
\end{equation}
Notice also that
$$ {\mathcal{H}}^{n-1} \big(F\cap\{x_1=-1/2\}\big)+
{\mathcal{H}}^{n-1} \big(F\cap\{x_1=1/2\}\big)+
{\mathcal{H}}^{n-1} \big((\partial F)\cap\{|x_1|<1/2\}\big)
={\mathcal{H}}^{n-1} (\partial F).$$
This, \eqref{0sde4hhhhajs-1} and~\eqref{0sde4hhhhajs-2}
give the proof
of~\eqref{9cidfvdhvb}.

Now we observe that
\begin{eqnarray*}
&& \sum_{k\in\{-1,+1\}}
\int_{F}\int_{F} \frac{dx\,dy}{|x-y+ke_1|^{n+s}} \\
&=&
\sum_{k\in\{-1,+1\}} \left[
\int_{F_\lambda}\int_{F_\lambda} \frac{dx\,dy}{|x-y+ke_1|^{n+s}}
+
\int_{F}\int_{F\setminus F_\lambda} \frac{dx\,dy}{|x-y+ke_1|^{n+s}}
+
\int_{F\setminus F_\lambda}\int_{F_\lambda}
\frac{dx\,dy}{|x-y+ke_1|^{n+s}} \right] .
\end{eqnarray*}
Thus, we exploit~\eqref{9ucdiwedf}
and~\eqref{9cidfvdhvb} and we obtain that
\begin{equation}\label{6sdcvf4rfasj}
\lim_{s\nearrow1}(1-s)
\sum_{k\in\{-1,+1\}}
\int_{F}\int_{F} \frac{dx\,dy}{|x-y+ke_1|^{n+s}}
=\omega_{n-1}\left[
{\mathcal{H}}^{n-1} (\partial F)-
{\mathcal{H}}^{n-1} \big((\partial F)\cap\{|x_1|<1/2\}\big)
\right].\end{equation}
Now we observe that if~$|k|\ge2$ and~$x$, $y\in S$, then
$$ |x-y+ke_1|\ge |k|-|x-y|\ge |k|-1\ge\frac{k}{2},$$
therefore, for any~$s\in[1/2,\,1)$,
$$ \sum_{k\in\Z\setminus\{0,-1,+1\}}
\int_{F}\int_{F} \frac{dx\,dy}{|x-y+ke_1|^{n+s}}
\le \sum_{k\in\Z\setminus\{0,-1,+1\}} \frac{2^{n+s}|F|^2}{|k|^{n+s}}\le
\sum_{k\in\Z\setminus\{0,-1,+1\}} \frac{2^{n+1}|F|^2}{|k|^{n+(1/2)}},$$
which implies that
$$ \lim_{s\nearrow1}(1-s) \sum_{k\in\Z\setminus\{0,-1,+1\}}
\int_{F}\int_{F} \frac{dx\,dy}{|x-y+ke_1|^{n+s}}
=0.$$
This and~\eqref{6sdcvf4rfasj} yield that
\begin{equation}\label{XC8dfgjj90}
\lim_{s\nearrow1}(1-s) \sum_{k\in\Z\setminus\{0\}}
\int_{F}\int_{F} \frac{dx\,dy}{|x-y+ke_1|^{n+s}} =
\omega_{n-1}\left[
{\mathcal{H}}^{n-1} (\partial F)-
{\mathcal{H}}^{n-1} \big((\partial F)\cap\{|x_1|<1/2\}\big)
\right].\end{equation}
Moreover, by Theorem~1 in~\cite{CV},
we have that
$$ \lim_{s\nearrow1}(1-s) {\rm Per}_s(F)=
\omega_{n-1} {\rm Per}(F)=\omega_{n-1}{\mathcal{H}}^{n-1}(\partial F).$$
Using this,~\eqref{90} and~\eqref{XC8dfgjj90}, we conclude that
\begin{eqnarray*}
&& \omega_{n-1}{\mathcal{H}}^{n-1}(\partial F)-
\lim_{s\nearrow1}(1-s) P_S(F)=
\lim_{s\nearrow1}(1-s) \Big[{\rm Per}_s(F)-P_S(F)\Big]\\
&&\quad=\lim_{s\nearrow1}(1-s)
\sum_{{k\in\Z\setminus\{0\}}}
\int_{F} \int_{F}
\frac{dx\,dy}{|x-y+ke_1|^{n+s}}\\&&\quad=
\omega_{n-1}\left[
{\mathcal{H}}^{n-1} (\partial F)-
{\mathcal{H}}^{n-1} \big((\partial F)\cap\{|x_1|<1/2\}\big)
\right],
\end{eqnarray*}
which gives the desired result.
\end{proof}

\section{Symmetric rearrangements in $x'$}\label{S rad}

Here we show that spherical rearrangements in the variable~$x'\in\R^{n-1}$
make our functional decrease.
Given a set~$A\subseteq\R^{n-1}$
(respectively, a function $f:\R^{n-1}\to[0,+\infty]$), we consider
its radially symmetric-decreasing rearrangement~$A^*$
(respectively~$f^*$,
see e.g. pages 80--81 in~\cite{lieb} for basic definition and properties).
Given~$A\subseteq\R^{n}$
(respectively, $f:\R^n\to\R$), fixed any~$x_1\in\R$ we
denote by~$A^{x_1,*}$
(respectively, $f^{x_1,*}:\R^{n-1}\to\R$) 
the radially symmetric-decreasing rearrangement of the set
$$ A^{x_1}:= \{ x'\in\R^{n-1} {\mbox{ s.t. }} (x_1,x')\in A\}$$
(respectively, of the function~$f(x_1,\cdot)$).

Given~$A\subseteq S$ we also set
$$ A^\star := \bigcup_{x_1\in[-1/2,\,1/2]} \{x_1\}\times A^{x_1,*}.$$

We now show that~$P_S$ decreases under
this radially symmetric-decreasing rearrangement in the variable~$x'$:

\begin{proposition}\label{SY1}
For any~$F\subseteq S$ with~$|F|<+\infty$, 
we have that $P_S(F^\star)\le P_S(F)$.
\end{proposition}

\begin{proof} Fix~$M\in\N$ (to be taken arbitrarily large in what follows).
We take~$h_M$ as in~\eqref{h M}.
Notice that, by Corollary~\ref{DEA-1} (and up to renaming~$C_M$),
$$ \kappa_M:=
\int_S h_M(x)\,dx
\le C_M\, \int_{\R^{n-1}} \min\{1,|x'|^{1-n-s}\}\,dx'\le C_M <+\infty.$$
Furthermore, $K(x_1+1,x')=K(x_1,x')$. This implies that
also the map~$x_1\mapsto h_M(x_1,x')$ is $1$-periodic for any
fixed~$x'\in\R^{n-1}$. Thus we can consider its integral on a period,
and we have that, for any~$r\in\R$,
$$ \int_{r+[-1/2,\,1/2]} h_M(x_1,x')\,dx_1
=\int_{[-1/2,\,1/2]} h_M(x_1,x')\,dx_1.$$
So, if we integrate over~$x'\in\R^{n-1}$, we obtain that
$$ \int_{r+S} h_M(x)\,dx = \int_S h_M(x)\,dx=\kappa_M.$$
Now, given any~$y\in\R^n$, we notice that~$-y+S=-y_1+S$, and thus
\begin{equation}\label{DAR}
\int_{S} h_M(x-y)\,dx = \int_{-y+S} h_M(x)\,dx=\kappa_M.
\end{equation}

Moreover, fixed~$x_1\in\R$, we have that the map~$\R^{n-1}\ni x'
\mapsto K(x_1,x')$ is
radially symmetric and decreasing, therefore $K^{x_1,*}(x')=K(x_1,x')$.
Accordingly, $h_M^{x_1,*}(x')=h_M(x_1,x')$.
Also, for any fixed~$x_1\in\R$, we have that~$\chi_{F^{x_1}}^*
=\chi_{(F^{x_1})^*}=\chi_{F^{x_1,*}}$.
Thus, fixed~$x_1\in\R$, we use the Riesz rearrangement inequality
(see e.g. Theorem~3.7 in~\cite{lieb}) and we obtain
\begin{eqnarray*}
&&\int_{\R^{n-1}}\int_{\R^{n-1}} \chi_{F}(x_1,x')
h_M(x_1-y_1,x'-y') \chi_{F}(y_1,y') \,dx'\,dy'
\\&=& \int_{\R^{n-1}}\int_{\R^{n-1}}  \chi_{F^{x_1}}(x') 
h_M(x_1-y_1,x'-y') \chi_{F^{y_1}}(y') \,dx'\,dy' \\
&\le& \int_{\R^{n-1}}\int_{\R^{n-1}}  \chi_{F^{x_1}}^*(x') 
h_M^*(x_1-y_1,x'-y') \chi_{F^{y_1}}^*(y') \,dx'\,dy'
\\ &=&
\int_{\R^{n-1}}\int_{\R^{n-1}}  \chi_{F^{x_1,*}}(x')
h_M(x_1-y_1,x'-y') \chi_{F^{y_1,*}}(y') \,dx'\,dy'\\
&=& \int_{\R^{n-1}}\int_{\R^{n-1}}  \chi_{F^\star}(x_1,x')
h_M(x_1-y_1,x'-y') \chi_{F^\star}(y_1,y') \,dx'\,dy'.
\end{eqnarray*}
Now we integrate over~$x_1\in[-1/2,1/2]$ and~$y_1\in[-1/2,1/2]$
and we obtain that
\begin{equation}\label{DAR2} \int_{S}\int_{S} \chi_{F}(x)
h_M(x-y) \chi_{F}(y) \,dx\,dy
\le \int_{S}\int_{S} \chi_{F^\star}(x)
h_M(x-y) \chi_{F^\star}(y) \,dx\,dy.\end{equation}
On the other hand, if~$x\in S$, we have that~$\chi_F(x)=1-\chi_{S\setminus F}(x)$,
therefore
\begin{eqnarray*}
&& \int_{S}\int_{S} \chi_{F}(x)
h_M(x-y) \chi_{F}(y) \,dx\,dy\\ &=&
\int_{S}\int_{S} 
h_M(x-y) \chi_F(y) \,dx\,dy
-\int_{S}\int_{S} \chi_{S\setminus F}(x)
h_M(x-y) \chi_{F}(y) \,dx\,dy \\
&=& \kappa_M\,|F|-\int_{S\setminus F}\int_{F} 
h_M(x-y) \,dx\,dy,
\end{eqnarray*}
thanks to~\eqref{DAR}. Similarly, we have that
$$ \int_{S}\int_{S} \chi_{F^\star}(x)
h_M(x-y) \chi_{F^\star}(y) \,dx\,dy=
\kappa_M\,|F^\star|-\int_{S\setminus F^\star}\int_{F^\star} 
h_M(x-y) \,dx\,dy.$$
Therefore, recalling~\eqref{DAR2}, we obtain that
$$ \kappa_M\,|F|-\int_{S\setminus F}\int_{F}
h_M(x-y) \,dx\,dy
\le \kappa_M\,|F^\star|-\int_{S\setminus F^\star}\int_{F^\star}
h_M(x-y) \,dx\,dy.$$
Hence, using that~$|F^\star|=|F|$ and that~$h_M\le K$, we obtain that
\begin{equation}\label{DAR3}
\int_{S\setminus F^\star}\int_{F^\star}
h_M(x-y) \,dx\,dy\le
\int_{S\setminus F}\int_{F}
K(x-y) \,dx\,dy.\end{equation}
Now we observe that, by Fatou Lemma,
$$ \liminf_{M\to+\infty} \int_{S\setminus F^\star}\int_{F^\star}
h_M(x-y) \,dx\,dy\ge
\int_{S\setminus F^\star}\int_{F^\star}
K(x-y) \,dx\,dy,$$
thus we can pass to the limit~\eqref{DAR3}
and obtain the desired result.
\end{proof}

In the light of Proposition~\ref{SY1}, we have that the cylindrical
symmetry assumption for the set of competitors in~${\mathcal{K}}$
(recall the definition on page~\pageref{456scdfvg}) can be weakened.
Indeed, it is not necessary to suppose that the competitors
are a priori cylindrically symmetric, since the cylindrical rearrangement
makes the energy functional decrease. It would be interesting to
weaken also the assumption that the set is a priori decreasing
with respect to~$x_1\in[0,\,1/2]$. In principle, a periodic
version of the cylindrical rearrangement should
prove that the energy also decreases under monotone rearrangement
in the $x_1$ variable. Though this property is in accordance
with the intuition and with some numerical simulations, it is not immediate to
give a rigorous proof of it, due to the presence of competing terms
in the sum that defines the functional, so we leave this as an open problem.

\end{document}